\documentclass[reqno,11pt]{amsart}
\usepackage{amssymb, color}
\allowdisplaybreaks[1]
\newtheorem{thm}{Theorem}[section]

\newtheorem{lemma}[thm]{Lemma}
\newtheorem{prop}[thm]{Proposition}

\newtheorem{rmk}[thm]{Remark}

\theoremstyle{definition}
\numberwithin{equation}{section}
\newcommand{\rr}{\mathbb{R}}

\newcommand{\eps}{\varepsilon}

\newcommand{\Uea}{U_{\eps,a}}
\newcommand{\Ueaz}{U_{\eps,a,z}}

\newcommand{\bG}{\bar{G}}
\newcommand{\Weaz}{W_\eps}
\newcommand{\ome}{\omega_\eps}
\newcommand{\diam}{\operatorname{diam}}
\newcommand{\ke}{k_\eps}
\newcommand{\We}{W_\eps}
\renewcommand{\ae}{a_{\eps,1}}
\newcommand{\aae}{a_{\eps,2}}
\newcommand{\toe}{\widetilde\omega_\eps}
\newcommand{\ue}{u_\eps}
\newcommand{\ga}{\gamma}
\newcommand{\epn}{\eps_n}
\newcommand{\un}{u_n}
\newcommand{\Ln}{L_{\eps_n}}
\newcommand{\Qn}{Q_{\eps_n}}
\newcommand{\zn}{z_{1,n}}
\newcommand{\znn}{z_{2,n}}
\newcommand{\tun}{\widetilde u_{1,n}}
\newcommand{\Hg}{\mathcal H_\gamma}
\newcommand{\bh}{\bar h}
\newcommand{\zg}{z_{1}^\gamma}
\newcommand{\zgg}{z_{2}^\gamma}
\newcommand{\Ie}{I_\eps}
\newcommand{\Ke}{K_\eps}
\newcommand{\Fe}{F_\eps}
\newcommand{\mM}{\mathcal M}
\newcommand{\zmin}{\underline z}
\newcommand{\tilzgg}{\widetilde z_2^\gamma}
\newcommand{\pp}{\overline p}
\begin{document}
\title[A free boundary problem related to plasmas]{Sign-changing two-peak solutions 
for an elliptic free boundary problem related to confined plasmas}
\author[G.~Pisante]{Giovanni Pisante}
\address[Giovanni Pisante]{Dipartimento di Matematica e Fisica, Seconda Universit\`{a} degli Studi di Napoli, 
Viale Lincoln 5, 81100 Caserta, Italy}
\email{giovanni.pisante@unina2.it}
\author[T.~Ricciardi]{Tonia Ricciardi$^*$}
\address[Tonia Ricciardi]{Dipartimento di Matematica e Applicazioni ``R.~Caccioppoli", Universit\`{a} di Napoli Federico II, 
Via Cintia, 80126 Napoli, Italy}
\email{tonricci@unina.it}
\thanks{$^*$Corresponding author}
\thanks{This research was carried out within the Marie Curie project FP7-MC-2009-IRSES-247486.
T.R.\ was also supported by PRIN {\em Aspetti variazionali e perturbativi nei problemi differenziali nonlineari}
and by the GNAMPA Project {\em Alcuni aspetti di equazioni ellittiche nonlineari.}  
G.P.\ was also supported by PRIN {\em Calcolo delle Variazioni} and by the GNAMPA Project {\em Un approccio variazionale all'analisi di modelli competitivi non lineari}.}
\begin{abstract}
By a perturbative argument, we construct solutions for a plasma-type problem with two 
opposite-signed sharp peaks at levels $1$ and $-\ga$, respectively, where $0<\ga<1$.\\
We establish some physically relevant qualitative properties for such solutions,
including the connectedness of the level sets
and the asymptotic location of the peaks as $\ga\to0^+$.
\end{abstract}
\date{December 23, 2015}
\maketitle
\section{Introduction}
Motivated by the description of equilibrium states for plasmas in a tokamak 
\cite{Te}, we consider the 
following problem:
\begin{equation}
\label{eq:pb1}
\left\{
\begin{aligned}
-\eps^2\Delta u=&\,(u-1)_+-(-u-\gamma)_+&&\hbox{in}\ \Omega,\\
u =&\,0&&\hbox{on}\ \partial \Omega,
\end{aligned}\right.
\end{equation}
where $\gamma$ is a positive constant, $\eps>0$ is a small positive number and $ \Omega\subset\rr^2$ is a smooth bounded domain.
Problem~\eqref{eq:pb1} generalizes the classical plasma problem:
\begin{equation*}
\left\{
\begin{aligned}
-\eps^2\Delta u=&\,(u-1)_+&&\hbox{in}\ \Omega,\\
u =&\,0&&\hbox{on}\ \partial \Omega,
\end{aligned}\right.
\end{equation*}
as derived in \cite{Te} and extensively analyzed in
\cite{BerestyckiBrezis, CaffarelliFriedman, kinderlehrerspruck, Te}, to the case of a nonlinearity indefinite in sign. 
From the physical point of view, problem~\eqref{eq:pb1} corresponds to the case where the tokamak contains two 
plasmas ionized with charges 1 and $-\ga$, respectively. 
The unknown function $u$ corresponds to the magnetic potential and $\eps>0$ is a constant depending 
on the constitution of the plasmas.
Problem~\eqref{eq:pb1} admits a variational characterization. Indeed, solutions to \eqref{eq:pb1}
correspond to critical points for the functional
\begin{equation*}
\Ie(u)=\frac{\eps^2}{2}\int_\Omega|\nabla u|^2-\frac{1}{2}\int_\Omega[(u-1)_+]^2
-\frac{1}{2}\int_\Omega[(-u-\gamma)_+]^2,
\qquad u\in H_0^1(\Omega).
\end{equation*}
Our first aim in this article is to construct sign-changing solutions with exactly two sharp peaks, 
via a perturbative Lyapunov-Schmidt argument as developed in \cite{CaoPengYan, DancerYan}.
Our next aim, containing the more innovative aspects, is to derive some new qualitative properties of solutions, 
including the connectedness of the level sets and
the asymptotic location of the peaks as $\gamma\to0^+$.  
\par
We recall that problems of the form $-\Delta u=f(u)$ in $\Omega$, $u=0$ on $\partial\Omega$
are also of central interest in the context of steady incompressible Euler flows, 
see, e.g., \cite{CaoLiuWei, CaoPengYan2015, SmVS} and the references therein.
In this context, such sign-changing solutions are related to the
Mallier-Maslowe counter rotating vortices \cite{BarPis, MallierMaslowe, PistoiaRicciardi}. 
Sign-changing solutions with exactly two peaks were also constructed for a
two-dimensional elliptic problem with large exponent in \cite{EspositoMussoPistoia2007}.
\par
Unlike the above mentioned models, problem~\eqref{eq:pb1} involves the nonlinearity $f(t)=(t-1)_+-(-t-\gamma)_+$
which is only Lipschitz continuous.
Consequently, major technical difficulties arise due to
the \lq\lq free boundaries" $\{u=1\}$ and $\{u=-\gamma\}$, particularly
in establishing some regularity properties that are essential in order to obtain solutions as critical points of 
continuously differentiable functionals.
Such difficulties were overcome in \cite{CaoPengYan} in the case of one-sided peak solutions.
A key ingredient to this end is to establish that such free boundaries have zero measure,
see Proposition~\ref{prop:zeromeasure} below.
Here, we shall obtain this key property by a new simple {\em ad hoc} argument involving the Faber-Krahn inequality, 
suitably tailored to 
the case of two-peak sign-changing solutions. 
We note that our perturbative construction does not provide any variational characterization
of solutions, which is often employed in establishing the zero measure of free boundaries.
Our arguments also suggest relations to the \lq\lq twisted eigenfunctions", recently analyzed in \cite{CHP, CHP1}.
\par
After constructing the peak solutions in Theorem~\ref{thm:main}, we analyze the asymptotic location 
of the peaks as $\ga\to0^+$. Roughly speaking, our result states that for the physically relevant solutions
corresponding to minima of the Kirchhoff-Routh Hamiltonian, as $\ga\to0^+$,
the positive peak approaches a harmonic center $\zmin\in\Omega$ and the negative peak escapes to a point
$\pp\in\partial\Omega$ which maximizes the outward normal derivative of the Green's function with
pole at $\zmin$. 
This type of property was introduced in \cite{PistoiaRicciardi}, where the case of a convex domain $\Omega$
is considered. Here, we remove the convexity assumption on $\Omega$, see Theorem~\ref{thm:properties}.
\par
This article is organized as follows. In Section~\ref{sec:results}
we introduce some notation and we precisely state the main results.
In Section~\ref{sec:ansatz} we define the ansatz for solutions in terms of the basic cell functions $U_{\eps,a,z}$
defined in \eqref{def:Ueaz}.
We also establish some necessary estimates for the approximate solutions.
Section~\ref{sec:reduction} contains the linear theory and the Lyapunov-Schmidt reduction 
by which we obtain a solution $\ue$ to a \lq\lq projected problem" 
for every fixed choice of peak points $z_1,z_2\in\Omega$,
$z_1\neq z_2$. Thus, we reduce problem~\eqref{eq:pb1}
to a four-dimensional problem. The results in this section rely on an approach introduced \cite{CaoPengYan, DancerYan}.
Therefore, some of the proofs in this section are only outlined for the reader's convenience.
In Section~\ref{sec:freebdry} we prove the crucial zero-measure property
for the free boundaries $\{\ue=1\}$ and $\{\ue=-\ga\}$, where $\ue$ is the solution to the
projected problem obtained in Section~\ref{sec:reduction}.
In Section~\ref{sec:redfunctl} we insert $\ue$ into the variational functional
for \eqref{eq:pb1} and we check that critical points for the resulting
function $\Ke(z_1,z_2)$ yield the desired solutions to the full equation 
\eqref{eq:pb1}; we also establish the connectedness of the level sets  thus establishing Theorem~\ref{thm:main}.
In Section~\ref{sec:qualprops} we analyze the limit profile 
of the \lq\lq minimal" solutions as $\ga\to0^+$,
as stated in Theorem~\ref{thm:properties}.
%
\section{Statement of the main results}
\label{sec:results}
In order to state our results precisely, we introduce some notation.
Let $s>0$ be defined by $s^2=j_0^{(1)}$, where $j_0^{(1)}\cong2.405$ is the first zero of $J^0$, the first Bessel function
of the first kind. Then, the first eigenvalue of $-\Delta$ in $B_s(0)$ subject to zero Dirichlet boundary conditions
is $\lambda_1(B_s(0))=1$.
Let $\varphi_1>0$ be the first eigenfunction of $-\Delta$ in $B_s(0)$ satisfying $\varphi_1(0)=1$. 
Namely, $\varphi_1$ satisfies:
\begin{equation*}
\left\{
\begin{aligned}
-\Delta\varphi_1=&\;\varphi_1&&\hbox{in\ }B_s(0)\\
\varphi_1=&\;0&&\mbox{on\ }\partial B_s(0)\\
\varphi_1(0)=&\;1.
\end{aligned}
\right.
\end{equation*}
For simplicity, in what follows we identify $\varphi$ with its radial profile;
namely, we denote $\varphi(s)=\varphi(x)\Big|_{|x|=s}=J^0(s)$.
\par
We denote by $U$ the \lq\lq basic cell function" defined on the whole space $\mathbb R^2$ by
\begin{equation}
\label{def:U}
U(x)=
\left\{
\begin{aligned}
&\varphi_1(x), &&\hbox{if\ }0\le|x|<s\\
&s|\varphi'(s)|\ln\frac{s}{|x|},&&\hbox{if\ }|x|\ge s.
\end{aligned}
\right.
\end{equation}
We note that any constant multiple of $U$ satisfies the equation
\begin{equation*}
-\Delta u=u_+
\qquad\hbox{ on }\;\mathbb R^2.
\end{equation*}
We shall obtain the desired peak solutions as perturbations of
an approximate solution defined in terms of suitably rescaled translations of $U$.
More precisely, let $R>0$ be a sufficiently large constant so that 
\[
R>\diam\Omega
\]
and consequently $\Omega\subset B_R(z)$ for all $z\in\Omega$.
We define, for any $\eps,a>0$: 
\begin{equation*}
U_{\eps,a}(x)=a+\frac{a}{\ln\frac{R}{s\eps}}U\left(\frac{x}{\eps}\right).
\end{equation*}
Then $U_{\eps,a}$ satisfies the Dirichlet problem
\begin{equation*}
\left\{
\begin{aligned}
-\eps^2\Delta u=&\;(u-a)_+&&\hbox{in\ }B_R(0)\\
u=&\;0&&\hbox{on\ }\partial B_R(0).
\end{aligned}
\right.
\end{equation*}
For any $z\in\mathbb R^2$ we set
\begin{equation}
\label{def:Ueaz}
U_{\eps,a,z}(x)=U_{\eps,a}(x-z).
\end{equation}
Thus, we obtain the family of cell functions $\{U_{\eps,a,z}:\ \eps,a>0,\ z\in\mathbb R^2\}$.
We note that as $\eps\to0$ the function $U_{\eps,a,z}$ develops a sharp peak at the point $z$
at level $a$. Moreover, as $\eps\to0$, we have $U_{\eps,a,z}\to0$ in $L_{\mathrm{loc}}^\infty(\Omega\setminus \{z\})$ 
and $U_{\eps,a,z}(z)\to a$.
\par
We denote by
$\bG$ the Green's function defined by
\begin{equation}
\label{def:bG}
\left\{
\begin{aligned}
-\Delta_x\bG(x,z)=&\;2\pi\delta_z && \hbox{in\ }\Omega\\
\bG(x,z)=&\;0 && \hbox{on\ }\partial\Omega,
\end{aligned}
\right.
\end{equation}
and by $g(x,z)$ the regular part of $\bG$, defined by
\begin{equation*}
\left\{
\begin{aligned}
-\Delta_x g(x,z)=&\;0 &&\hbox{in\ }\Omega\\
g(x,z)=&\;\ln\frac{R}{|x-z|} && \hbox{on\ }\partial\Omega.
\end{aligned}
\right.
\end{equation*}
We note that $\bG(x,z)=\ln\frac{R}{|x-z|}-g(x,z)$ and, by the choice of $R$, we have
$g(x,z)\ge0$ for all $x,z\in\Omega$.
We denote by $h:\Omega\to\mathbb R$ the Robin's function:
\begin{equation*}
h(z)=g(z,z).
\end{equation*}
We remark that with this notation we have 
\begin{equation}
\label{eq:hproperties}
h(z)\ge0\ \hbox{for all\ }z\in\Omega,
\qquad
h(z)\to+\infty\ \hbox{as\ }z\to\partial\Omega.
\end{equation}
We set 
\begin{equation*}
\mathcal M=\{(x,y)\in\Omega\times\Omega:x\neq y\}.
\end{equation*}
We denote by $\mathcal H_\gamma$ the Kirchhoff-Routh type Hamiltonian (see, e.g., \cite{SmVS}) defined 
by
\begin{equation*}
\mathcal H_\gamma(z_1,z_2)=h(z_1)+2\gamma\bG(z_1,z_2)+\gamma^2h(z_2)
\end{equation*}
for $(z_1,z_2)\in\mathcal M$.
We note that $\mathcal H_\gamma$ is bounded from below on $\mathcal M$.
We denote by $P:H^1(\Omega)\to H_0^1(\Omega)$ the standard projection operator and
we denote by $\mathrm{cat}\mathcal M$ the Lusternik-Schnirelmann category of $\mathcal M$.
\par
Our aim is to establish the following results.
\begin{thm}
\label{thm:main}
For every $\ga>0$ there exists $\eps_\ga>0$ such that for all $0<\eps<\eps_\ga$ problem~\eqref{eq:pb1}
admits at least $\mathrm{cat}\mathcal M$ sign-changing two-peak solutions
of the form 
\begin{equation}
\label{eq:solform}
u_\eps^i=PU_{\eps,a_{\eps,1}^i,z_{\eps,1}^i}-PU_{\eps,a_{\eps,2}^i,z_{\eps,2}^i}+\omega_\eps^i, 
\end{equation}
$i=1,\ldots,\mathrm{cat}\mathcal M$,
where $\|\omega_\eps^i\|_{L^\infty(\Omega)}=O(\eps/|\ln\eps|)$, $a_{\eps,1}^i=1+O(|\ln\eps|^{-1})$,
$a_{\eps,2}^i=\gamma+O(|\ln\eps|^{-1})$ and $(z_{\eps,1}^i,z_{\eps,2}^i)\in\mathcal M$,
$(z_{\eps,1}^i,z_{\eps,2}^i)\to(z_1^i,z_2^i)\in\mathcal M$,
with $(z_1^i,z_2^i)$ a critical point for $\mathcal H_\gamma$.
\par
Moreover, the following properties hold.
\begin{enumerate}
\item [(i)]
The level sets $\{u_{\eps}^i>1\}$ and $\{u_\eps^i<-\gamma\}$,
$i=1,\ldots,\mathrm{cat}\mathcal M$, are connected.
\item[(ii)]
If $\gamma=1$, then problem~\eqref{eq:pb1}
admits at least $\mathrm{cat}[\mathcal M/((x,y)\sim(y,x))]$ pairs of sign-changing solutions
of the form~\eqref{eq:solform}, $i=1,\ldots,\mathrm{cat}[\mathcal M/((x,y)\sim(y,x))]$.
\end{enumerate}
\end{thm}
It may be worth observing that, although $L^\infty(\Omega)$-bounded, the family of solutions $\{\ue\}_{\eps>0}$ obtained in
Theorem~\ref{thm:main}
presents a lack of compactness in $C^\alpha$-sense.
\par
In order to state our second result, we recall that $\zmin\in\Omega$ is a harmonic center of $\Omega$
if it is a minimum point of the Robin's function $h(z)=g(z,z)$, see \cite{BF}. 
If $\Omega$ is convex, then $\zmin$ is unique \cite{CaffarelliFriedman1985}.
\begin{thm}
\label{thm:properties}
For every $\gamma>0$, let $u_\eps^\ga$ be a family of solutions to \eqref{eq:pb1},
as obtained in Theorem~\ref{thm:main},
with peak points at $(z_{1,\eps}^\ga,z_{2,\eps}^\ga)\to(z_1^\ga,z_2^\ga)$
as $\eps\to0^+$,
where $\Hg(z_1^\ga,z_2^\ga)=\min_\mathcal M\Hg$.
For every $\eta>0$ there exist $\ga_\eta>0$ and $\eps(\ga_\eta)>0$
such that $d(z_{1,\eps}^\ga,\zmin)<\eta$, $d(z_{2,\eps}^\ga,p_0)<\eta$
for all $\ga\in(0,\ga_\eta)$, $\eps\in(0,\eps(\ga_\eta))$,
where $\zmin$ is a harmonic center for $\Omega$ and $\pp\in\partial\Omega$
satisfies
$\partial_\nu G(\zmin,\pp)=\max_{p\in\partial\Omega}\partial_\nu G(\zmin,p)$.
\end{thm}
Theorem~\ref{thm:properties} implies that for peak solutions corresponding to minima of $\Hg$,
for small values of $\gamma$ 
the positive peak is approximately located at a harmonic center $\zmin\in\Omega$ and the negative peak is near the boundary $\partial\Omega$.
\par
\subsection*{Notation}
Henceforth, for any measurable set $S\subset\mathbb R^2$ we denote by $mS$ the two-dimensional Lebesgue measure
of $S$ and by $1_S$ the characteristic function of $S$.
All integrals are taken with respect to the Lebesgue measure; when the integration variable is clear from
the context, we may omit it.
We denote $\|\cdot\|_p=\|\cdot\|_{L^p(\Omega)}$.
We denote by $C>0$ a general constant, whose actual value may vary from line to line.  Let $\Omega \subset \mathbb R^2$ be an open set, for a given function $V$ defined on $\Omega \times \Omega$ and $Z=(z_1,z_2)\in \Omega \times \Omega$  we use the notation $\frac{\partial V}{\partial z_{i,j}}$ to denote the partial derivative of $V$ with respect to the component $j$ of the variable $i$. 
\section{Ansatz and properties of approximate solutions} 
\label{sec:ansatz}
The aim of this section is to define suitable approximate solutions $\We$ 
to problem~\eqref{eq:pb1}
in terms of the basic cell functions $U_{\eps,a,z}$, for $\eps,a>0$, $z\in\mathbb R^2$, 
defined in \eqref{def:Ueaz}. We also establish some basic properties of $\We$
which will be needed in the sequel.
\par
We recall form \cite{CaoPengYan} that for any $a>0$ and $\eps>0$, the problem 
\begin{equation}
\label{eq:bubblepb}
\left\{
\begin{aligned}
-\eps^2\Delta u=&\;(u-a)_+&&\hbox{in\ }B_R(0)\\
u=&\;0&&\hbox{on\ }\partial B_R(0)
\end{aligned}
\right.
\end{equation}
has as unique solution $U_{\eps,a}(x)$ defined by
\begin{equation}
\label{def:bubble}
U_{\eps,a}(x)= a+\frac{a}{\ln\frac{R}{s\eps}}U\left(\frac{x}{\eps}\right)=
\left\{
\begin{aligned}
&a\left[1+\ke\varphi_1\left(\frac{x}{\eps}\right)\right]&&\hbox{if\ }|x|\le s\eps\\
&\frac{a}{\ln\frac{R}{s\eps}}\ln\frac{R}{|x|}&&\hbox{if\ }s\eps\le|x|\le R,
\end{aligned}
\right.
\end{equation} 
where 
\[
\ke:=\frac{1}{s\varphi_1'(s)\ln\frac{s\eps}{R}}>0.
\]
We note that $\Uea>0$ in $B_R(0)$ and $\ke\to0$ as $\eps\to0$.
Moreover, we have that $\Uea(x)\to0$ in $L_{\mathrm{loc}}^\infty(B_R(0)\setminus\{0\})$.
Correspondingly, we obtain the family $\{\Ueaz:\ z\in\Omega,\ \eps,a>0\}$ of 
peak solutions defined in \eqref{def:Ueaz}.
Since the functions $\Ueaz(x)$ do not satisfy the zero Dirichlet boundary condition on $\partial \Omega$, 
as usual we define their projections on $H^1_0(\Omega)$. We recall that, given $u\in H^1(\Omega)$, 
the projection  of $u$ into $H^1_0(\Omega)$, denoted $Pu$, is the unique weak solution to
\[
\Delta Pu=\Delta u\quad \hbox{in}\ \Omega ,\qquad  Pu=0\quad \hbox{on}\ \partial\Omega .
\]
In particular, in view of \eqref{eq:bubblepb}, $P\Ueaz(x)$ satisfies
\begin{equation*}
\left\{
\begin{aligned}
-\eps^2\Delta P\Ueaz =&\;(\Ueaz-a)_+ &&\hbox{in\ }\Omega\\
u=&\;0 &&\hbox{on\ }\partial \Omega.
\end{aligned}
\right.
\end{equation*}
Moreover, for all $\eps>0$ sufficiently small so that $B_{s\eps}(z)\subset\Omega$, we have
that $\Ueaz$ coincides with a Newtonian potential on $\partial\Omega$, namely
\[
\Ueaz(x)=\frac{a}{\ln\frac{R}{s\eps}}\ln\frac{R}{|x-z|} \;\; \forall\, x\in \partial \Omega.
\]
The following is readily checked.
\begin{lemma}
\label{lem:proj-exp}
Let  $\eps>0$ be such that $B_{s \eps}(z)\subset \Omega$.
We have, for every $z\in\Omega$:
\begin{equation}
	\label{PU-represent}
	P\Ueaz(x)=\Ueaz(x)-\frac{a}{\ln\frac{R}{s\eps}}g(x,z)\;\;,\;\; \forall x\, \in  \, \Omega
\end{equation}
and more precisely
\begin{equation}
\label{PU-represent-2}
P\Ueaz(x)=
\left\{
\begin{aligned}
& a\left[1+\ke\varphi_1\left(\frac{x-z}{\eps}\right)-\frac{g(x,z)}{\ln\frac{R}{s\eps}}\right],&&\mbox{if\ }x\in B_{s\eps}(z)\\
&\frac{a}{\ln\frac{R}{s\eps}}\bG(x,z),&&\mbox{if\ }x\in \Omega\setminus B_{s\eps}(z).
\end{aligned}
\right.
\end{equation}
\end{lemma}
We seek solutions whose form is approximately the difference of two
cell functions of the form \eqref{def:Ueaz}.
Let $Z=(z_1,z_2)\in\mathcal M$.
We make the following
\par\medskip
\textbf{Ansatz.} The solution $u$ to problem \eqref{eq:pb1} is of the form:
\begin{equation}
\label{eq:ansatz}
\begin{aligned}
u:=&\;W_\eps+\ome\\
W_\eps:=&\;PU_{\eps,z_1,a_1}-PU_{\eps,z_2,a_2}.
\end{aligned}
\end{equation}
\medskip
Henceforth, we denote $U_i=U_{\eps,z_i,a_i}$, $i=1,2$.
We assume that
\begin{align}
\label{assump:Z}
&d(z_i,\partial\Omega)\ge\delta>0,\ i=1,2;
&&|z_1-z_2|\ge\delta
\end{align}
for some $\delta>0$. 
\par\medskip
\textbf{Choice of $a_i=a_{\eps,i}$, $i=1,2$}.
The pair of constants $(a_{\eps,1},a_{\eps,2})$ is chosen as the (unique) solution to the following linear system:
\begin{equation}
\label{eq:asystem}
\left\{
\begin{aligned}
\Big(1-\frac{g(z_1,z_1)}{\ln\frac{R}{s\eps}}\Big)a_1-\frac{\bG(z_1,z_2)}{\ln\frac{R}{s\eps}}a_2=&1\\
-\frac{\bG(z_1,z_2)}{\ln\frac{R}{s\eps}}a_1+\Big(1-\frac{g(z_2,z_2)}{\ln\frac{R}{s\eps}}\Big)a_2=&\gamma.
\end{aligned}
\right.
\end{equation}
Namely,
\begin{equation*}
\begin{aligned}
a_{\eps,1}=&\;\frac{1-\frac{g(z_2,z_2)}{\ln\frac{R}{s\eps}}-\bG(z_1,z_2)\gamma}
{(1-\frac{g(z_1,z_1)}{\ln\frac{R}{s\eps}})(1-\frac{g(z_2,z_2)}{\ln\frac{R}{s\eps}})-\frac{\bG(z_1,z_2)^2}{(\ln\frac{R}{s\eps})^2}}
=1+O\left(\frac{1}{|\ln\eps|}\right)\\
a_{\eps,2}=&\;\frac{(1-\frac{g(z_1,z_1)}{\ln\frac{R}{s\eps}})\gamma-\bG(z_1,z_2)\gamma}
{(1-\frac{g(z_1,z_1)}{\ln\frac{R}{s\eps}})(1-\frac{g(z_2,z_2)}{\ln\frac{R}{s\eps}})-\frac{\bG(z_1,z_2)^2}{(\ln\frac{R}{s\eps})^2}}
=\gamma+O\left(\frac{1}{|\ln\eps|}\right).
\end{aligned}
\end{equation*}
\par
\medskip
In order to justify the choice of $a_i=a_{\eps,i}$, $i=1,2$ in terms of $\gamma,\eps,z_1,z_2$
and $\Omega$,
we consider the ``error term"
\[
\ell_\eps(x):=\eps^2\Delta\Weaz+(\Weaz-1)_+-(-\Weaz-\gamma)_+
\]
and we estimate $\ell_\eps$ near $z_1$ and near $z_2$. Let us first observe that in order to approximate a solution of \eqref{eq:pb1} it is reasonable to choose $a_1 \simeq 1$ and $a_2 \simeq \gamma$  therefore we suppose $0< a_i < K$ for a given $K>0$.
\par
In $B_{s\eps}(z_1)$, in view of Lemma~\ref{lem:proj-exp}, we have
\begin{align*}
\ell_\eps(x)=&-(U_1-a_1)_++\Big(U_1-\frac{a_1}{\ln\frac{R}{s\eps}}g(x,z_1)
-\frac{a_2}{\ln\frac{R}{s\eps}}\bG(x,z_2)-1\Big)_+\\
&\quad+(U_2-a_2)_+-\Big(\frac{a_2}{\ln\frac{R}{s\eps}}\bG(x,z_2)-PU_1-\gamma\Big)_+.
\end{align*}
We note that $(U_2-a_2)_+\equiv0$ in $B_{s\eps}(z_1)$. Moreover, for sufficiently small $\eps>0$
we also have $\frac{a_2}{\ln\frac{R}{s\eps}}\bG(x,z_2)-\gamma\le0$ in $B_{s\eps}(z_1)$. Since $PU_1\ge0$
in $\Omega$, we conclude that 
\[
\left(\frac{a_2}{\ln\frac{R}{s\eps}}\bG(x,z_2)-PU_1-\gamma\right)_+\equiv0
\]
in $B_{s\eps}(z_1)$. 
It follows that in $B_{s\eps}(z_1)$ we have
\[
\ell_\eps(x)=-(U_1-a_1)_++\left(U_1-\frac{a_1}{\ln\frac{R}{s\eps}}g(x,z_1)
-\frac{a_2}{\ln\frac{R}{s\eps}}\bG(x,z_2)-1\right)_+.
\]
Taking $x=z_1$ we fit $a_1$ by requiring that $\ell_\eps(z_1)=0$.
We consequently derive
\[
a_1=\frac{a_1}{\ln\frac{R}{s\eps}}g(z_1,z_1)
+\frac{a_2}{\ln\frac{R}{s\eps}}\bG(z_1,z_2)+1,
\]
which yields the first equation in \eqref{eq:asystem}.
The second equation is obtained similarly.
\par
\medskip
We note that the Ansatz and the choice of $(a_{\eps,1},a_{\eps,2})$ show that the interaction between 
$PU_1$ and $PU_2$ is essentially negligible. 
Moreover, we have the following expansions
\begin{align}
	\label{exp:ai}
	\frac{\partial a_{\eps,i}}{\partial z_{i,h}}=&\;O \left(\frac{1}{|\log\eps|} \right),\\
	\nonumber
	\frac{\partial U_{\eps,z_i,a_{\eps,i}}(x)}{\partial z_{i,h}}=&
\;\frac{a_{\eps,i}}{\ln\frac{R}{s\eps}}\frac{\partial}{\partial z_{i,h}}U\left(\frac{x-z_i}{\eps}\right)
+\frac{U(\frac{x-z_i}{\eps})}{\ln\frac{R}{s\eps}}\frac{\partial a_{\eps,i}}{\partial z_{i,h}}\\
\label{exp:ueps}
=&\begin{cases}
		\displaystyle \frac{k_\eps a_{\eps,i}}{\eps  } \varphi'_1\left( \frac{|x-z_{i}|}{\eps} \right)\frac{z_{i,h}-x_h}{|x-z_{i}|}+O \left(\frac{\partial a_{\eps,i}}{\partial z_{i,h}} \right),\ &\hbox{if\ }x\in B_{s \eps}(z_i) \\ 
		\displaystyle \frac{k_\eps a_{\eps,i}}{|x-z_{i}|}\frac{z_{i,h}-x_h}{|x-z_{i}|}+O \left(\frac{\partial a_{\eps,i}}{\partial z_{i,h}} \frac{\log\frac{R}{|x-z_{i}|}}{\log\frac{R}{s \eps}} \right),&\hbox{if\ } x\in \Omega \setminus B_{s \eps}(z_i)
	\end{cases},
\end{align}
where $U$ is defined in \eqref{def:U}.
In view of \eqref{eq:asystem} we have the following.
\begin{lemma}
\label{lem:Wid}
The following identities hold:
\[
W_\eps(x)\equiv 
\begin{cases}
	1 +a_1\ke\varphi_1\left(\frac{x-z_1}{\eps}\right) -\frac{a_1}{\ln\frac{R}{s\eps}}[g(x,z_1)-g(z_1,z_1)]-\frac{a_2}{\ln\frac{R}{s\eps}}[\bG(x,z_2)-\bG(z_1,z_2)],& \hbox{in}\, B_{s\eps}(z_1); \\
	-\gamma-a_2\ke\varphi_1(\frac{x-z_2}{\eps})+\frac{a_1}{\ln\frac{R}{s\eps}}[\bG(x,z_1)-\bG(z_2,z_1)]
+\frac{a_2}{\ln\frac{R}{s\eps}}[g(x,z_2)-g(z_2,z_2)],&\hbox{in}\, B_{s\eps}(z_2);\\
\frac{a_1-a_2}{\ln\frac{R}{s\eps}}\bG(z_1,z_2)+\frac{a_1}{\ln\frac{R}{s\eps}}[\bG(x,z_1)-\bG(z_1,z_2)]-\frac{a_2}{\ln\frac{R}{s\eps}}[\bG(x,z_2)-\bG(z_1,z_2)], &\hbox{otherwise}.
\end{cases}
\]
\end{lemma}
\begin{proof}
We exploit the explicit expressions of $PU_1,PU_2$ and the definitions of $a_1,a_2$
as in \eqref{eq:asystem}.
\par
In $B_{s\eps}(z_1)$ we have:
\begin{align*}
W_\eps(x)=&\;PU_1(x)-PU_2(x)
=a_1\left[1+\ke\varphi_1\left(\frac{x-z_1}{\eps}\right)-\frac{g(x,z_1)}{\ln\frac{R}{s\eps}}\right]
-\frac{a_2}{\ln\frac{R}{s\eps}}\bG(x,z_2)\\
=&\;a_1\ke\varphi_1\left(\frac{x-z_1}{\eps}\right)+a_1\left[1-\frac{g(x,z_1)}{\ln\frac{R}{s\eps}}\right]-\frac{a_2}{\ln\frac{R}{s\eps}}\bG(x,z_2)\\
=&\;a_1\ke\varphi_1\left(\frac{x-z_1}{\eps}\right)+a_1\left[1-\frac{g(z_1,z_1)}{\ln\frac{R}{s\eps}}\right]-\frac{a_2}{\ln\frac{R}{s\eps}}\bG(z_1,z_2)\\
&\qquad-\frac{a_1}{\ln\frac{R}{s\eps}}\left[g(x,z_1)-g(z_1,z_1)\right]-\frac{a_2}{\ln\frac{R}{s\eps}}\left[\bG(x,z_2)-\bG(z_1,z_2)\right]\\
=&\;1+a_1\ke\varphi_1\left(\frac{x-z_1}{\eps}\right)-\frac{a_1}{\ln\frac{R}{s\eps}}[g(x,z_1)-g(z_1,z_1)]-\frac{a_2}{\ln\frac{R}{s\eps}}\left[\bG(x,z_2)-\bG(z_1,z_2)\right],
\end{align*}
where we used the first equation in \eqref{eq:asystem} in the last inequality.
\par
Similarly, in $B_{s\eps}(z_2)$ we have:
\begin{align*}
W_\eps(x)
=&\;\frac{a_1}{\ln\frac{R}{s\eps}}\bG(x,z_1)-a_2\left[1+\ke\varphi_1\left(\frac{x-z_2}{\eps}\right)-\frac{g(x,z_2)}{\ln\frac{R}{s\eps}}\right]\\
=&\;-a_2\ke\varphi_1\left(\frac{x-z_2}{\eps}\right)-a_2\left[1-\frac{g(x,z_2)}{\ln\frac{R}{s\eps}}\right]+\frac{a_1}{\ln\frac{R}{s\eps}}\bG(x,z_1)\\
=&\;-a_2\ke\varphi_1\left(\frac{x-z_2}{\eps}\right)-a_2\left[1-\frac{g(z_2,z_2)}{\ln\frac{R}{s\eps}}\right]+\frac{a_1}{\ln\frac{R}{s\eps}}\bG(z_2,z_1)\\
&\qquad+\frac{a_2}{\ln\frac{R}{s\eps}}\left[g(x,z_2)-g(z_2,z_2)\right]+\frac{a_1}{\ln\frac{R}{s\eps}}\left[\bG(x,z_1)-\bG(z_2,z_1)\right]\\
=&\;-\gamma-a_2\ke\varphi_1\left(\frac{x-z_2}{\eps}\right)
+\frac{a_2}{\ln\frac{R}{s\eps}}\left[g(x,z_2)-g(z_2,z_2)\right]+\frac{a_1}{\ln\frac{R}{s\eps}}\left[\bG(x,z_1)-\bG(z_2,z_1)\right],
\end{align*}
where we used the second equation in  \eqref{eq:asystem} in the last inequality.
\par
The third identity readily follows observing that in $\Omega\setminus\left(B_{s\eps}(z_1)\cup B_{s\eps}(z_2)\right)$
we have
\[
W_\eps(x)=\frac{a_1}{\ln\frac{R}{s\eps}}\bG(x,z_1)-\frac{a_2}{\ln\frac{R}{s\eps}}\bG(x,z_2).
\]
\end{proof}
The choice of $a_1,a_2$ in \eqref{eq:asystem} also implies the following estimates.
\begin{lemma}
\label{1.10-CPY}
For any $L>0$ fixed constant we have, for $\eps$ sufficiently small, the following expansions for $\We$:
\begin{equation}
\label{expansion-1}
\begin{split}
W_\eps(y)-1=&\; U_{\eps,z_1,a_1} - a_1-\frac{a_1}{\log \frac{R}{s\eps}} \langle D g(z_1,z_1) , y-z_1 \rangle  + \frac{a_2}{\log \frac{R}{s\eps}} \langle D \bar G(z_1,z_2) , y-z_1 \rangle \\
 & + O\left( \frac{\eps^2}{|\log\eps|} \right)
\end{split}
\end{equation}
for $y\in B_{L\eps}(z_1)$;
\begin{equation}
\label{expansion-gamma}
\begin{split}
-W_\eps(y)-\gamma= &\; U_{\eps,z_2,a_2} - a_2-\frac{a_2}{\log \frac{R}{s\eps}} \langle D g(z_2,z_2) , y-z_2 \rangle  - \frac{a_1}{\log \frac{R}{s\eps}} \langle D \bar G(z_2,z_i) , y-z_2 \rangle \\
 & + O\left( \frac{\eps^2}{|\log\eps|} \right)
\end{split}
\end{equation}
for $y\in B_{L\eps}(z_2)$.
\end{lemma}
\begin{proof} 
The proof is readily derived from Lemma~\ref{lem:Wid} and a Taylor expansion at $z_1$ and $z_2$.
\end{proof}
Using Lemma~\ref{1.10-CPY} we now provide a quantitative estimate of the sets where 
the approximate solution $W_\eps$ takes values less than $-\gamma$, between $-\gamma$ and $1$ and greater than $1$. 
These estimates will be useful in the study of the linearized problem for the finite dimensional reduction scheme.
\begin{lemma}[Level set estimates] 
\label{lemma-livelli} 
There exist $T\gg1$ and $0<\sigma < 1$ independent on $\eps$ and $0<\eps_0\ll1$ such that for all $0<\eps<\eps_0$ we have
\begin{equation*}
\begin{aligned}
W_\eps(y)-1 >&\;0 \; , &&y\in B_{s \eps (1-T\eps)}(z_1)\,, \\
-W_\eps(y)-\gamma>&\;0 \; , &&y\in B_{s \eps (1-T\eps)}(z_2)\,\\
-\gamma <W_\eps(y) <&\;1  \; , \, &&y\in \Omega \setminus \left( B_{s \eps (1+\eps^\sigma)}(z_1) \cup B_{s \eps (1+\eps^\sigma)}(z_2) \right).
\end{aligned}
\end{equation*}
\end{lemma}
\begin{proof}
We follow the proof of Lemma~A.1 in \cite{CaoPengYan}.
We start by taking $y\in B_{s \eps (1-T\eps)}(z_1)$. Using \eqref{def:bG}, the monotonicity of $\varphi_1$, a Taylor expansion of $\varphi_1$ around $s$ and $\varphi'_1(s)<0$, we deduce that
\[
\begin{split}
	U_{\eps,z_1,a_1}(y) - a_1 & = \frac{1}{\log\frac{R}{s\eps}} \left( \frac{- a_1}{s\varphi'_1(s)}\varphi_1\left( \frac{|y-z_1|}{\eps}\right)\right) \geq  \frac{1}{\log\frac{R}{s\eps}} \left( \frac{- a_1}{s\varphi'_1(s)}\varphi_1\left( s-sT\eps\right)\right) \\ 
	&  = \frac{1}{\log\frac{R}{s\eps}} \left( \frac{a_1}{s\varphi'_1(s)}[\varphi'_1(s)sT\eps+ O(\eps^2)]\right) =  \frac{a_1 T \eps}{\log\frac{R}{s\eps}} + O\left( \frac{\eps^2}{|\log\eps|} \right).
\end{split}
\]
Choose $\eps \ll 1$ such that $\gamma_1= \min\{1/2,\gamma/2\} \leq a_i \leq \max\{3/2,\gamma+1/2\}= \gamma_2$ and let $M=M(\delta,\gamma)>0$ such that for any $(z_1,z_2)$ satisfying \eqref{assump:Z}, we have
\[
M \geq a_1 |Dg(z_1,z_2)| +a_2 |D \bar G(z_1,z_2)|.
\]
Using \eqref{expansion-1} and the previous estimates we deduce that
\[
\begin{split}
W_\eps(y)-1= &\; U_{\eps,z_1,a_1}(y) - a_1 +  \frac{1}{\log \frac{R}{s\eps}} a_2 \langle D \bar G(z_1,z_2) , y-z_1 \rangle  \\
& - \frac{1}{\log \frac{R}{s\eps}} a_1 \langle D g(z_1,z_1) , y-z_1 \rangle  + O\left( \frac{\eps^2}{|\log\eps|} \right)  \\
\geq 	& \; \frac{1}{\log \frac{R}{s\eps}}  \left( a_1 T \eps - M s\eps (1-T\eps) \right) +   O\left( \frac{\eps^2}{|\log\eps|} \right)  \\
\geq 	& \; \frac{\eps}{\log \frac{R}{s\eps}}  \left( \gamma_1 T - M s \right) +   O\left( \frac{\eps^2}{|\log\eps|} \right)
\end{split}
\]
and the claim follows by choosing $T > \frac{Ms}{\gamma_1}$ and $\eps_0$ sufficiently small.

If $y\in B_{s \eps (1-T\eps)}(z_2)$ we proceed in a similar way using \eqref{expansion-gamma} and the fact that 
\[
U_{\eps,z_2,a_2}(y) - a_2 \geq \frac{a_2 T \eps}{\log\frac{R}{s\eps}} + O\left( \frac{\eps^2}{|\log\eps|} \right).
\] 

To get the estimate of $W_\eps$ away from the points $z_i$ we begin by noticing that, if $y\in \Omega \setminus \left( B_{\eps^{\bar\sigma}}(z_1) \cup B_{\eps^{\bar\sigma}}(z_2) \right)$, for some $\bar\sigma < 1$ to be chosen later, by Lemma \ref{lem:proj-exp}, we can write, for any $\eps$ sufficiently small such that $\eps^{\bar \sigma}> s \eps$,
\[
PU_{\eps,z_i,a_i}(y)= \frac{a_i}{\log\frac{R}{s\eps}} \bar G(y,z_1)=\frac{a_i}{\log\frac{s\eps}{R}} \log\frac{|y-z_i|}{R} + O \left(\frac{1}{|\log\eps|} \right). 
\]
It follows that, by choosing for instance $\bar \sigma < \gamma_1/2 \gamma_2$ and $\eps$ sufficiently small, we can write
\[
\begin{split}
W_\eps(y)-1& = \frac{a_1}{\log\frac{s\eps}{R}} \log\frac{|y-z_1|}{R} -  \frac{a_2}{\log\frac{s\eps}{R}} \log\frac{|y-z_2|}{R} -1 + O \left(\frac{1}{|\log\eps|} \right) \\
& \leq 2\gamma_2 \frac{\log\frac{\eps^{\bar\sigma}}{R}}{\log\frac{s \eps}{R}} -1  + O \left(\frac{1}{|\log\eps|} \right)  = \gamma_2 \bar \sigma \frac{\log\eps}{\log\frac{s \eps}{R}} -1  + O \left(\frac{1}{|\log\eps|} \right) < 0 
\end{split}
\]
and analogously 
\[
\begin{split}
-W_\eps(y)-\gamma & \leq 2\gamma_2 \frac{\log\frac{\eps^{\bar\sigma}}{R}}{\log\frac{s \eps}{R}} -\gamma  + O \left(\frac{1}{|\log\eps|} \right)  = \gamma_2 \bar \sigma \frac{\log\eps}{\log\frac{s \eps}{R}} -\gamma  + O \left(\frac{1}{|\log\eps|} \right) < 0.
\end{split}
\]
Moreover, for $y$ in the annulus $B_{\eps^{\bar\sigma}}(z_1)\setminus B_{s\eps (1+T\eps^{\bar\sigma})}(z_1)$, using \eqref{PU-represent} for $PU_{\eps,z_1,a_1}(y)$, \eqref{PU-represent-2} for $PU_{\eps,z_2,a_2}(y)$ and \eqref{eq:asystem}, we can write, for sufficiently small $\eps$,
\[
\begin{split}
	W_\eps(y)-1 = & \; U_{\eps,z_1,a_1}(y) - \frac{a_1}{\log\frac{R}{s\eps}}g(y,z_1) -   \frac{a_2}{\log\frac{R}{s\eps}}\bar G(y,z_2) -1 \\
				= &  \; U_{\eps,z_1,a_1}(y) - a_1 - \frac{a_1}{\log\frac{R}{s\eps}}\big( g(y,z_1) - g(z_1,z_1) \big) - \frac{a_2}{\log\frac{R}{s\eps}} \big( \bar G(y,z_2)- \bar G(z_1,z_2)\big) \\
				\leq & \; \frac{a_1}{\log\frac{s\eps}{R}} \log\frac{|y-z_1|}{R} -a_1  + 2M\gamma_2 \frac{\eps^{\bar\sigma}}{\log\frac{R}{s\eps}} \\
				= & \;  2M\gamma_2 \frac{\eps^{\bar\sigma}}{\log\frac{R}{s\eps}} - \frac{a_1}{\log\frac{R}{s\eps}} \log\frac{|y-z_1|}{s\eps} \leq   2M\gamma_2 \frac{\eps^{\bar\sigma}}{\log\frac{R}{s\eps}}  - a_1 \frac{\log(1+T \eps^{\bar\sigma})}{\log\frac{R}{s\eps}} \\
				\leq & \;  \frac{\eps^{\bar\sigma}}{\log\frac{R}{s\eps}} \left(2M\gamma_2 -\gamma_1 \frac{T}{2}\right)  
\end{split}
\]
We can choose $T>\frac{4M\gamma_2}{\gamma_1}$ in order to assure that the last term in the previous inequality is negative. In a similar way we can choose $T$ sufficiently large to have, for any $y \in B_{\eps^{\bar\sigma}}(z_2)\setminus B_{s\eps (1+T\eps^{\bar\sigma})}(z_2)$, 
\[
-W_\eps(y)-\gamma \leq 2M\gamma_2 \frac{\eps^{\bar\sigma}}{\log\frac{R}{s\eps}} - \frac{a_2}{\log\frac{R}{s\eps}} \log\frac{|y-z_2|}{s\eps} < 0.
\]
Finally, to get the claim, we choose $\sigma < \bar \sigma$ and $\eps_0$ small enough to satisfy the previous estimates and $B_{s\eps (1+T\eps^{\bar\sigma})}(z_i)\subset B_{s\eps(1+\eps^\sigma)}(z_i)$.

\end{proof}

\section{Reduction to a four-dimensional problem}
\label{sec:reduction}
Our aim in this section is to reduce problem~\eqref{eq:pb1} to a four-dimensional problem
depending on $Z=(z_1,z_2)$, via a Lyapunov-Schmidt argument;
that is, we shall solve a \lq\lq projection" of problem~\eqref{eq:pb1}
for any given $Z=(z_1,z_2)\in\mathcal M$ satisfying \eqref{assump:Z},
see Proposition~\ref{prop:redstatement} below.
The content of this section follows \cite{CaoPengYan, DancerYan} closely; therefore,
some proofs are only outlined.
\par
Henceforth, we set
\[
V_{\eps,Z,j}:=PU_{\eps,z_j,a_j},
\]
where $a_j=a_j(\eps,Z)$, $j=1,2$ are defined in \eqref{eq:asystem}. 
With this notation, ansatz~\eqref{eq:ansatz} takes the form
\[
u= V_1-V_2+\omega_\eps
\]
with the convention that $V_j:=V_{\eps,Z,j}$, $j=1,2$.
\par
We first introduce some function spaces.
For $p>1$ let
\begin{equation}
\begin{aligned}
F_{\varepsilon,Z}=&\left\{v\;:\; v\in L^p(\Omega) \,,\; 
\int_{\Omega} \frac{\partial V_j}{\partial z_{j,h}}v=0 \,,\; j,h=1,2\right\}\\
E_{\varepsilon,Z}=&\left\{ u \;: \; u\in W^{2,p}(\Omega)\cap H_0^1(\Omega) \,,\; 
\int_{\Omega} \Delta\left(\frac{\partial V_j}{\partial z_{j,h}}\right)u=0 \,,\;j,h=1,2\right\}.
\end{aligned}
\end{equation}
We define a \lq\lq localized projection operator" $Q_\varepsilon:L^p(\Omega)\to F_{\varepsilon,Z}$
whose action is supported in $B_{2s\eps}(z_1)\cup B_{2s\eps}(z_2)$.
To this end, let $\xi(t):\mathbb R\to\mathbb R$ be a smooth cut-off function
satisfying $\xi(t)\equiv1$ for $0\le t\le 1$, $\xi(t)\equiv0$ for $t\ge2$, $0\le\xi(t)\le1$ and define for $i = 1,2$
\[
\xi^i_{\eps}(y):= \xi\left( \frac{|y-z_i|}{s \varepsilon} \right) .
\]
For every $u\in L^p(\Omega)$ we set
\begin{equation}
\label{def:Qe}
Q_\varepsilon u(y)=u(y)- \xi^1_{\eps}(y) \big \langle \mathbf{b}_1 , \nabla_{z_1} U_1(y) \big\rangle - \xi^2_{\eps}(y)\big  \langle \mathbf{b}_2 ,\nabla_{z_2} U_2(y)\big \rangle 
\end{equation}
where the vectors $\mathbf{b}_i=(b_{i,1},b_{i,2})\in\mathbb R^2$ are defined as solution of the linear system 
\[
\mathbf A  \cdot (\mathbf{b}_1,\mathbf{b}_2) = \int_\Omega u \cdot  (\nabla_{z_1}V_1,\nabla_{z_2}V_2)
\]
where the $4\times 4$ matrix $\mathbf A$ is given by
\[
\mathbf A := \int_\Omega \big(\nabla_{z_1}V_1,\nabla_{z_2}V_2\big) \otimes \big( \xi^1_{\eps} \nabla_{z_1}  U_1 , \xi^2_{\eps} \nabla_{z_2} U_2 \big ) \, dy 
\]
Equivalently, in coordinates \eqref{def:Qe} takes the form
\[
Q_\varepsilon u(y)=u(y)-\sum_{j,h=1}^2 b_{j,h}\,\xi^j_{\eps}(y)  
\frac{\partial U_{\varepsilon , z_j, a_{\varepsilon,j}} }{\partial z_{j,h}}(y).
\]
Let us remark that the previous system is solvable since, in view of 
the expansions \eqref{exp:ai}--\eqref{exp:ueps}, the elements of the matrix $\mathbf A$, $A_{i,j,h,k}$ for $i,j,h,k \in \{1,2\}$, 
satisfy the following orthogonality properties:
\begin{equation}
\label{eq:bjhorthogonal}
\begin{split}
A_{i,j,h,k} = & \int_\Omega\xi^j_\eps\frac{\partial U_{\eps,z_j,a_{\eps,j}}}{\partial z_{j,h}} 
\frac{\partial V_{\eps,Z,a_{\eps,i}}}{\partial z_{i,k}}\, dy \\
= &  \int_{B_{2s\eps(z_j)}}\xi^j_\eps\frac{\partial U_{\eps,z_j,a_{\eps,j}}}{\partial z_{j,h}}
\frac{\partial U_{\eps,Z,a_{\eps,i}}}{\partial z_{i,k}}\, dy+O\left(\frac{\eps}{|\ln\eps|^2}\right)\\
=&\; c'\delta_{ij}\delta_{hk}\frac{1}{\left(\log\frac{R}{s\eps}\right)^2}+O\left(\frac{\eps}{|\ln\eps|^2}\right),
\end{split}
\end{equation}
where 
\[
c'=\frac{\pi}{(s\phi'_1(s))^2}\left( \int_0^{s}(\phi'_1(t))^2\,dt + \int_s^{2s}\frac{\xi(t)}{t^2}\,dt\right) >0
\] is a constant. 
\par
With this notation, our main result in this section is the following.
\begin{prop}
\label{prop:redstatement}
For every fixed $Z=(z_1,z_2)\in\Omega\times\Omega$ satisfying \eqref{assump:Z} there exists $\eps_0>0$ 
with the property that for all $0<\eps<\eps_0$ there exists $\omega_\eps\in E_{\eps,Z}$ with 
\[
\|\omega_\eps\|_{L^\infty(\Omega)}\le\|\omega_\eps\|_\infty=O\left(\frac{\eps}{|\ln\eps|}\right)
\] 
and such that the function $u_\eps:= \We+\omega_\eps$ is a solution of the \lq\lq projected problem":
\begin{equation}
\label{eq:projpb}
Q_\eps\left[\eps^2\Delta u_\eps+(u_\eps-1)_+-(-u_\eps-\ga)_+\right]=0
\qquad\mbox{in\ }\Omega.
\end{equation}
\end{prop}
In view of Proposition~\ref{prop:redstatement}, the construction
of a peak solution to \eqref{eq:pb1} is reduced to finding $Z=(z_1,z_2)$
such that the solution $\ue$ to the projected problem~\eqref{eq:projpb}
is actually a solution to the full problem~\eqref{eq:pb1}.
\par
The remaining part of this section is devoted to the proof of Proposition~\ref{prop:redstatement}.
We note that the linearization of the problem
\[
-\eps^2\Delta u= (u-1)_+-(u+\gamma)_-\mbox{\ in\ }\Omega,\qquad u\in H_0^1(\Omega)
\] 
about a solution $u$ is given by
\[
-\eps^2\Delta \phi= \chi_{\{u>1\} \cup \{u<-\gamma\}} \phi \mbox{\ in\ }\Omega,\qquad \phi\in H_0^1(\Omega).
\]
Let $f_{\eps}:\Omega\to\mathbb R$ be defined by
\[
f_{\eps}(y)=
\left\{
\begin{aligned}
&1, &&\mathrm{if\ } W_\eps(y)-1>0 \,\textrm{ or } W_\eps(y)+\gamma <0 \\
&0, &&\mathrm{otherwise}.
\end{aligned}
\right.
\]
\begin{rmk}
From Lemma \ref{lemma-livelli} it follows that $f_{\eps}(y)=1$ in $B_{s \eps (1-T\eps)}(z_1) \cup B_{s \eps (1-T\eps)}(z_2)$ and $f_{\eps}(y)=0$ in $\Omega \setminus \left( B_{s \eps (1+\eps^\sigma)}(z_1) \cup B_{s \eps (1+\eps^\sigma)}(z_2) \right)$.	
\end{rmk}
We define $L_\eps: W^{2,p}(\Omega)\cap H_0^1(\Omega)\to L^p(\Omega)$ by setting
\[
L_\eps u=-\eps^2\Delta u-f_{\eps}(y)u.
\]
Following the contradiction argument used in the proof of \cite[Lemma 3.3]{CaoPengYan}, we have the following:
\begin{lemma}
\label{lem:inv}
Let $p>1$ be fixed. There are constants $c_0>0$ and $\eps_0>0$ such that
for any $\eps\in(0,\eps_0)$, $Z$ satisfying \eqref{assump:Z}, $u\in E_{\eps,Z}$
with $Q_\eps L_\eps u=0$ in $\Omega\setminus\cup_{j=1}^2B_{L\eps}(z_j)$ for some large $L>0$,
then
\begin{equation}
\label{eq:linestimate}
\|Q_\eps L_\eps u\|_{L^p(\Omega)}\ge c_0\eps^{2/p}\|u\|_{L^\infty(\Omega)}.
\end{equation}
Consequently, $Q_\eps L_\eps:E_{\eps,Z}\to F_{\eps,Z}$ is one-to-one and onto,
and for every $v\in F_{\eps,Z}$ we have
\[
\|(Q_\eps L_\eps u)^{-1}v\|_{L^\infty(\Omega)}\le c_0^{-1}\eps^{-2/p}\|v\|_{L^p(\Omega)}.
\]
\end{lemma}
\begin{proof}
The proof follows \cite{CaoPengYan} closely. We outline the main ideas for the reader's convenience.
We first establish \eqref{eq:linestimate}.
Arguing by contradiction, we assume that there exist $\epn\to0$, $Z_n=(z_{1,n},z_{2,n})$
satifying \eqref{assump:Z} and $\un\in E_{\epn,Z_n}$, $\|\un\|_\infty=1$,
such that $\Qn\Ln\un=0$ in $\Omega\setminus \cup_{j=1}^2B_{L\epn}(z_{j,n})$
and $\|\Qn\Ln\un\|_p\le\epn^{2/p}/n$.
Let $(\mathbf{b}_{1,n},\mathbf{b}_{2,n})\in\mathbb R^{2\times 2}$ be such that
\[
\Qn\Ln\un=\Ln\un-\mathcal B(\mathbf{b}_{1,n},\mathbf{b}_{2,n}).
\]
where we have defined
\[
\mathcal B(\mathbf{b}_{1,n},\mathbf{b}_{2,n}):=\sum_{j,h=1}^2b_{j,h,n}\xi\left(\frac{|y-z_{j,n}|}{s\epn}\right)\frac{\partial U_{j,n}}{\partial z_{j,h}}
\] 
Then, similarly as in \cite{CaoPengYan}, using Lemma \ref{lemma-livelli}, \eqref{exp:ai} \eqref{exp:ueps} and \eqref{eq:bjhorthogonal} , we derive
\[
b_{j,h,n}=O\left(\epn^{1+\sigma}|\ln\epn|\right).
\]
Consequently, in virtue of \eqref{exp:ueps}, we have
\[
\left\|\mathcal B(\mathbf{b}_{1,n},\mathbf{b}_{2,n})\right\|_{L^p(\Omega)}
=O\left(\epn^{\frac{2}{p}+\sigma}\right)
\]
and
\[
\|\Ln\un\|_{L^p(\Omega)}=\|\Qn\Ln\un\|_{L^p(\Omega)}+O(\epn^{\frac{2}{p}+\sigma})=o(\epn^{\frac{2}{p}}).
\]
Now, rescaling $u_n$ around $\zn$, we define
\[
\tun(y):=\un(\epn y+\zn), \qquad y\in\Omega_n,
\]
where $\Omega_n=\{y:\epn y+\zn\in\Omega\}$.
Since $\|\tun\|_\infty=\|\un\|_\infty=1$, we have that $\tun$
is bounded in $W_{\mathrm{loc}}^{2,p}(\mathbb R^2)$ and there exists $u_1$ such that
$\tun\to u_1$ in $C_{\mathrm{loc}}^2(\mathbb R^2)$.
On the other hand, in view of Lemma~\ref{lemma-livelli} we have that $f_{\epn}(\epn y+\zn)\to 1_{B_s(0)}$.
We conclude that $u_1$ satisfies the problem
\[
-\Delta u_1-1_{B_s(0)}u_1=0,
\qquad u_1\in L^\infty(\mathbb R^2).
\]
Now Proposition~3.1 in \cite{CaoPengYan} implies that
\[
u_1=c_1\frac{\partial U}{\partial x_1}+c_2\frac{\partial U}{\partial x_1},
\]
for some $c_1,c_2\in\mathbb R$,
where $U$ is the basic cell function given by \eqref{def:U}. On the other hand, taking limits in the orthogonality condition
\[
\int_\Omega\Delta\left(\frac{\partial V_{\epn,Z_n,i}}{\partial z_{i,n}}\right)\un=0,
\]
which holds true since $\un\in E_{\epn,Z_n}$, we derive $c_1=c_2=0$.
We conclude that for any $L>0$ 
\[
\|\un\|_{L^\infty(B_{L\eps}(\zn))}=o(1).
\] 
By a similar rescaling argument at $z_{2,n}$ we also obtain that
\[
\|\un\|_{L^\infty(B_{L\eps}(\znn))}=o(1).
\]
Now, recalling that $\Qn\Ln\un=0$ in $\Omega\setminus\cup_{j=1}^2B_{L\epn}(z_{j,n})$
we finally conclude that $\|\un\|_\infty=o(1)$ which is a contradiction. Hence, \eqref{eq:linestimate} is established.
\par
Now, we check that $Q_\eps L_\eps:E_{\eps,Z}\to F_{\eps,Z}$ is one-to-one and onto.
Let $u\in E_{\eps,Z}$ be such that $Q_\eps L_\eps u=0$. 
Since the action of $Q_\eps$ is localized in $B_{2s\eps}(z_1)\cup B_{2s\eps}(z_1)$,
we have $Q_\eps L_\eps u=L_\eps u=0$ in $\Omega\setminus(B_{L\eps}(z_1)\cup B_{L\eps}(z_1))$,
for any $L\ge 2s$. Now estimate~\eqref{eq:linestimate} yields $u=0$ and the asserted one-to-one property follows.
\par
We are left to check the surjectivity property. We note that by continuity of the projection operator $P$,
we have $\frac{\partial V_j}{\partial z_{j,h}}=P\frac{\partial U_j}{\partial z_{j,h}}\in H_0^1(\Omega)$,
$j,h=1,2$.
In particular, for any $u\in E_{\eps,Z}$ we have
\[
\int_\Omega\Delta u\frac{\partial V_j}{\partial z_{j,h}}
=\int_\Omega u\Delta \frac{\partial V_j}{\partial z_{j,h}}=0,
\qquad j,h=1,2.
\]
It follows that $Q_\eps\Delta u=\Delta u$ for all $u\in E_{\eps,Z}$ and it is easy to see that it is one to one and onto from $E_{\eps,Z}$ to $F_{\eps,Z}$. Now, let 
$v\in F_{\eps,Z}$. We check that the equation $Q_\eps L_\eps u=v$ admits a solution in $E_{\eps,Z}$. Equivalently, we seek a solution to
\[
\eps^2 u-(-Q_\eps\Delta)^{-1}[Q_\eps f_\eps(y)u]=(-Q_\eps\Delta)^{-1}v.
\]
The operator on the r.h.s.\ above is of the form $\eps^2I+$compact operator. Then, by the Fredholm alternative, we conclude that $Q_\eps L_\eps:E_{\eps,Z}\to F_{\eps,Z}$ is onto.
\end{proof}
We now define a fixed point problem which is equivalent to \eqref{eq:projpb}.
We recall that the error term $\ell_\eps=\ell_\eps(x)$
is defined by
\begin{equation*}
\begin{split}
\ell_\eps := & \,  \eps^2\Delta W_\eps+(W_\eps-1)_+-(-W_\eps-\gamma)_+\\
= & -(U_1-a_1)_++(U_2-a_2)_++(W_\eps-1)_+-(-W_\eps-\gamma)_+.
\end{split}
\end{equation*}
We define the higher order error $R_\eps$ as follows:
\[
R_\eps(\omega) :=(W_\eps+\omega-1)_+-(-W_\eps-\omega-\gamma)_+ - (W_\eps-1)_+-(-W_\eps-\gamma)_+-f_\eps(y)\omega. 
\]
We observe that if  $\omega\in W^{2,p}(\Omega)\cap H_0^1(\Omega)$ satisfies 
\[
L_\eps(\omega)=\ell_\eps + R_\eps(\omega)
\]
then $u=W_\eps+\omega$ is a solution to problem~\eqref{eq:pb1}. 
We note that a solution $\omega$ to
\begin{equation}
\label{eq:projpb2}
Q_\eps L_\eps (\omega) = Q_\eps (\ell_\eps + R_\eps(\omega))
\end{equation}
readily yields a solution to \eqref{eq:projpb}.
In view of the invertibility property of $Q_\eps L_\eps$, as stated in Lemma~\ref{lem:inv},
we define the operator $G_\eps:E_{\eps,Z}\to F_{\eps,Z}$ by setting
\[
G_\eps(\omega):=(Q_\eps L_\eps)^{-1}Q_\eps(\ell_\eps+R_\eps(\omega)).
\]
Then, the projected problem \eqref{eq:projpb2}
is equivalent to the fixed point problem 
\begin{equation}
\label{eq:fixpoint}
\omega=G_\eps(\omega).
\end{equation}
Arguing similarly as in the proof of \cite[Proposition 3.6]{CaoPengYan} one can prove that \eqref{eq:fixpoint} has a unique solution. 
More precisely, we have the following Lemma which together with the previous observations provides the proof of Proposition \ref{prop:redstatement}.
\begin{lemma}
\label{lem:fixpoint}
There exists an $\eps_0>0$ such that for any $\eps\in(0,\eps_0)$ and for any $Z$
satisfying \eqref{assump:Z}, equation \eqref{eq:fixpoint} has a unique solution
$\omega_\eps\in E_{\eps,Z}$ with
\begin{equation}
\label{eq:omeest}
\|\omega_\eps\|_\infty=O\left(\frac{\eps}{|\ln\eps|}\right).
\end{equation}
\end{lemma}
\begin{proof}
The proof follows \cite{CaoPengYan} closely. We outline it for the reader's convenience.
We define 
\[
M:=E_{\eps,Z}\cap\{\|\omega\|_\infty<\eps\}.
\] 
We start by checking that $G_\eps$ is a map from $M$ to $M$.
Since the action of $Q_\eps$ is localized in $B_{2s\eps}(z_1)\cup B_{2s\eps}(z_1)$,
we have 
\[
Q_\eps(\ell_\eps+R_\eps(\omega))=\ell_\eps+R_\eps(\omega)=0
\qquad\mbox{in\ }\Omega\setminus(B_{L\eps}(z_1)\cup B_{L\eps}(z_1)),
\]
for $L\ge2s$.
Therefore, we may use estimate~\eqref{eq:linestimate} to obtain
\[
\|G_\eps(\omega)\|_\infty=\|(Q_\eps L_\eps)^{-1}(Q_\eps\ell_\eps+Q_\eps R_\eps(\omega))\|_\infty
\le C\eps^{-2/p}\|Q_\eps\ell_\eps+Q_\eps R_\eps(\omega)\|_p.
\]
Similarly as in \cite{CaoPengYan}, we estimate
\begin{equation*}
\begin{aligned}
&\|Q_\eps\ell_\eps+Q_\eps R_\eps(\omega)\|_p\le C(\|\ell_\eps\|_p+\|R_\eps(\omega)\|_p)\\
&\|\ell_\eps\|_p\le\frac{C\eps^{2/p+1}}{|\ln\eps|}\\
&\|R_\eps(\omega)\|_p\le C\eps^{(2+\sigma)/p}\|\omega\|_\infty.
\end{aligned}
\end{equation*}
Since $\omega\in M$, we have $\|\omega\|_\infty\le\eps$ and we conclude that
\begin{equation}
\label{eq:Gest}
\|G_\eps(\omega)\|_\infty\le C\eps^{-2/p}\left(\frac{\eps^{2/p+1}}{|\ln\eps|}
+\eps^{(2+\sigma)/p}\|\omega\|_\infty\right)\le\frac{C\eps}{|\ln\eps|}.
\end{equation}
In particular, $G_\eps$ maps $M$ into $M$.
\par
By similar arguments, we can also check the contraction property:
\[
\|G_\eps(\omega_1)-G_\eps(\omega_2)\|_\infty\le C\eps^{-2/p}\|R_\eps(\omega_1)-R_\eps(\omega_2)\|_p
=o(1)\|\omega_1-\omega_2\|_\infty.
\]
It follows that for all sufficiently small $\eps>0$ there exists a fixed point $\omega_\eps\in M$
for $G_\eps$. Moreover, in view of \eqref{eq:Gest}, we obtain the asserted estimate
\[
\|\omega_\eps\|_\infty=\|G_\eps(\omega_\eps)\|_\infty\le\frac{C\eps}{|\ln\eps|}.
\]
\end{proof}
%
%
\begin{proof}[Proof of Proposition~\ref{prop:redstatement}]
Now the proof follows as  direct consequence of Lemma~\ref{lem:inv} and Lemma~\ref{lem:fixpoint}.
\end{proof}
\section{Free boundary properties for the projected problem}
\label{sec:freebdry}
In order to show that $(z_1,z_2)\to\omega_\eps$ is $C^1$ it is essential to show that
the free boundaries $\{\ue=1\}$ and $\{\ue=-\ga\}$ have two-dimensional Lebesgue measure
equal to zero, see the {\it{Step 2}} in the proof of Proposition~\ref{prop:redfunct} below.
The aim of this section is to establish such a property, via an argument involving the Faber-Krahn inequality.
Throughout this section, for any $S\subset\mathbb R^2$, 
we denote by $mS$ the two-dimensional Lebesgue measure of $S$.
%
The main result in this section is the following.
\begin{prop}
\label{prop:zeromeasure}
Let $\ue$ be the solution to the \lq\lq projected problem" \eqref{eq:projpb},
as obtained in Proposition~\ref{prop:redstatement}.
It holds that
\[
m\{u_\eps=1\}=0=m\{u_\eps=-\gamma\}.
\]
\end{prop}
Before proving the previous result, we observe that arguing exactly as in Lemma \ref{lemma-livelli}, in view of the decay estimate~\eqref{eq:omeest} for $\omega_\eps$, we can prove the following useful level set estimates.
\begin{lemma}
\label{lem:livellieps}
There exist $0<\eps_0\ll1$ and $T\gg1$ such that for all $0<\eps<\eps_0$ we have
\begin{equation*}
\begin{aligned}
W_\eps(y)+\ome-1 >&\;0 \; , &&y\in B_{s \eps (1-T\eps)}(z_1)\,, \\
-W_\eps(y)-\ome-\gamma>&\;0 \; , &&y\in B_{s \eps (1-T\eps)}(z_2)\,\\
-\gamma <W_\eps(y)+\ome <&\;1  \; , \, &&y\in \Omega \setminus \left( B_{s \eps (1+\eps^\sigma)}(z_1) 
\cup B_{s \eps (1+\eps^\sigma)}(z_2) \right).
\end{aligned}
\end{equation*}
where $\sigma >0$ is a small constant.
\end{lemma}
\begin{proof}[Proof of Proposition~\ref{prop:zeromeasure}]
Being a solution to the projected problem~\eqref{eq:projpb}, $\ue$ satisfies
\[
-\eps^2\Delta\ue=(\ue-1)_+-(-\ue-\gamma)_+
- \xi^1_{\eps}(y) \big \langle \mathbf{b}_1 , \nabla_{z_1} U_1(y) \big\rangle - \xi^2_{\eps}(y)\big  \langle \mathbf{b}_2 ,\nabla_{z_2} U_2(y)\big \rangle
\]
for some $\mathbf{b}_1, \mathbf{b}_2 \in \mathbb{R}^2$.
\par
We first show that $m\{u_\eps=1\}=0$.
Arguing by contradiction, suppose that $m\{u_\eps=1\}>0$.
In view of Lemma~\ref{lem:livellieps} we have $\{u_\eps=1\}\subset B_{2s\eps}(z_1)$.
In $\{u_\eps=1\}\subset B_{2s\eps}(z_1)$ we have $\Delta\ue\equiv0$ and therefore
$\ue$ satisfies, in  $\{u_\eps=1\}$
\[
0=-\eps^2\Delta\ue=(\ue-1)_+-\big \langle \mathbf{b}_1 , \nabla_{z_1} U_1(y) \big\rangle = -\big \langle \mathbf{b}_1 , \nabla_{z_1} U_1(y) \big\rangle
\]
Since $\partial U_1/\partial z_{1,1}$ and $\partial U_1/\partial z_{1,2}$
are linearly independent, we conclude that $\mathbf{b}_1=0$. 
In particular, $\ue$ satisfies
\[
-\eps^2\Delta\ue=(\ue-1)_+\ \mbox{in\ }B_{2s\eps}(z_1).
\]
Let
\[
\Omega_{\eps,1}:=\{\ue>1\}.
\]
\textit{Claim.} $\Omega_{\eps,1}$ has exactly one connected component.
\par
Indeed, suppose the contrary and let $\toe\subset\Omega_{\eps,1}$ be a connected component of
$\Omega_{\eps,1}$ with $z_1\not\in\toe$.
Note that $\toe$ is an open subset of $\Omega$.
Then, $\ue$ satisfies the problem
\begin{equation*}
\left\{
\begin{aligned}
-\eps^2\Delta\ue=&\;\ue-1
&&\hbox{in\ }\toe\\
\ue=&\;1
&&\hbox{on\ }\partial\toe.
\end{aligned}
\right.
\end{equation*}
Setting $v_\eps=\ue-1$, we derive
\begin{equation*}
\left\{
\begin{aligned}
-\eps^2\Delta v_\eps=&\;v_\eps
&&\hbox{in\ }\toe\\
v_\eps=&\;0
&&\hbox{on\ }\partial\toe.
\end{aligned}
\right.
\end{equation*}
Multiplying by $v_\eps$ and integrating, we obtain 
\[
\eps^2\int_{\toe}|\nabla v_\eps|^2=\int_{\toe}v_\eps^2\le\frac{1}{\lambda_1(\toe)}\int_{\toe}|\nabla v_\eps|^2
\]
and consequently 
\[
\lambda_1(\toe)\le\frac{1}{\eps^2}.
\]
On the other hand, setting $A_\eps:=B_{s\eps(1+\eps^\sigma)}\setminus B_{s\eps(1-T\eps)}$,
we note, again by Lemma~\ref{lem:livellieps}, that $\toe\subset A_\eps$ and $mA_\eps=2\pi s^2\eps^{2+\sigma}(1+o(1))$
as $\eps\to0$.
It follows that
\[
\lambda_1(\toe)\ge\lambda_1(A_\eps)\ge\lambda_1(A_\eps^\ast)=\frac{\pi j_0^2}{2\pi s^2\eps^{2+\sigma}(1+o(1))}
\]
where $A_\eps^\ast$ denotes a ball with measure $mA_\eps^\ast=mA_\eps$,
and where we used the Faber-Krahn inequality (see, e.g., \cite{Bandle})
to derive the last inequality. This is a contradiction.
Hence, $\Omega_{\eps,1}$ is connected and the claim is established. 
Arguing in a similar way we can prove that the set $\Omega_{\eps}^{-\ga}:=\{\ue<-\gamma\}$ is connected. 
Finally, we claim that by virtue of Lemma \ref{lem:livellieps}, the maximum principle 
implies that the set $\Omega_{\eps,-\ga}^1:=\{x \in \Omega \;: \; -\gamma< u_\eps(x) <1\} $ is also connected. 
Indeed, arguing by contradiction we suppose  that there exists a connected 
component $\hat\omega_{\eps}$ of $\Omega_{\eps,-\ga}^1$ with $\hat\omega_{\eps}\subset B_{2s\eps}(z_1)$. 
Then $u_\eps$ is a solution of
\[
\begin{cases}
	\Delta u_\eps = 0 & \textrm{in } \, \hat\omega_{\eps}\\
	u_\eps =1 & \textrm{on } \, \partial\hat\omega_{\eps}\\
	u_\eps < 1 & \textrm{in } \, \hat\omega_{\eps}, 
\end{cases}
\] 
a contradiction.
\par
Now we can apply Theorem~1 and Theorem~2 in \cite{HartmanWintner} together with the argument in
\cite{kinderlehrerspruck}, p. ~136,
to conclude that $\nabla\ue\neq0$ on $\{\ue=1\}$. Therefore, $\{\ue=1\}$ is a simple
$C^2$-curve. In particular $m\{\ue=1\}=0$, a contradiction.
\par
The remaining part of the statement is obtained similarly.
\end{proof}
\section{The reduced functional and the proof of Theorem~\ref{thm:main}}
\label{sec:redfunctl}
We recall that the Euler-Lagrange functional for \eqref{eq:pb1} is given by
\[
\Ie(u)=\frac{\eps^2}{2}\int_\Omega|\nabla u|^2-\frac{1}{2}\int_\Omega[(u-1)_+]^2
-\frac{1}{2}\int_\Omega[(-u-\gamma)_+]^2,
\qquad u\in H_0^1(\Omega).
\]
For every $z_1,z_2\in\Omega$ satisfying \eqref{assump:Z} we define the \lq\lq reduced functional"
\[
\Ke(z_1,z_2)=\Ie(W_\eps+\omega_\eps),
\]
where $\omega_\eps=\ome(Z)$ is the error function obtained in Proposition~\ref{prop:redstatement}.
Then, we have:
\begin{lemma}
\label{lem:functred}
If $Z=(z_1,z_2)$ is a critical point for $\Ke$, then $\ue$ defined by 
\[
\ue=PU_{\eps,a_{\eps,1}(Z),z_1}-PU_{\eps,a_{\eps,2}(Z),z_2}+\ome
\]
is a solution for problem~\eqref{eq:pb1}.
\end{lemma}
\begin{proof}
By construction, $\ue=\We+\ome$ satisfies the \lq\lq projected problem"
\begin{equation}
\label{eq:projected}
-\eps^2\Delta\ue-(\ue-1)_++(-\ue-\ga)_+=\sum_{j,h=1}^2b_{jh}\,\xi\left(\frac{|y-z_j|}{s\eps}\right)
\frac{\partial U_j}{\partial z_{j,h}},
\end{equation}
for some constants $b_{jh}$, $j,h=1,2$.
Equivalently, we have
\[
\langle\Ie'(\ue),\varphi\rangle
=\sum_{j,h=1}^2b_{jh}\int_{B_{2s\eps(z_j)}}\,\xi\left(\frac{|y-z_j|}{s\eps}\right)
\frac{\partial U_j}{\partial z_{j,h}}\,\varphi,
\]
for all $\varphi\in H_0^1(\Omega)$.
We verify that if $(z_1,z_2)$ is a critical point of $\Ke$, then the constants $b_{jh}$
in \eqref{eq:projected} vanish for all $j,h=1,2$.
Indeed, we have:
\begin{equation*}
\begin{aligned}
\frac{\partial\Ke(Z)}{\partial z_{j,h}}=\langle I'(\ue),\frac{\partial\ue}{\partial z_{j,h}}\rangle
=\sum_{i,\bh=1}^2b_{i,\bh}\int_{B_{2s\eps(z_i)}}\,\xi\left(\frac{|y-z_i|}{s\eps}\right)
\frac{\partial U_i}{\partial z_{i,\bh}}\left(\frac{\partial PU_1}{\partial z_{j,h}}-\frac{\partial PU_2}{\partial z_{j,h}}
+\frac{\partial\ome}{\partial z_{j,h}}\right).
\end{aligned}
\end{equation*}
In view of \eqref{exp:ueps} we conclude that $b_{jh}=0$, $j,h=1,2$.
\end{proof}
Hence, in order to conclude the proof of Theorem~\ref{thm:main} we are left to
obtain a critical point for $\Ke$. 
For $a_1,a_2>0$ let $\mathcal H_{a_1,a_2}$ be the Kirchhoff-Routh type Hamiltonian defined by
\[
\mathcal H_{a_1,a_2}(z_1,z_2)=a_1^2h(z_1)+2a_1a_2\bG(z_1,z_2)+a_2^2h(z_2).
\]
Recall that
\[
\mM=\{(z_1,z_2)\in\Omega\times\Omega:\ z_1\neq z_2\}.
\]
The following proposition clarifies the relation between $\Ke$ and $\mathcal H_{a_1,a_2}$:
\begin{prop}
\label{prop:redfunct}
The following expansions hold true as $\eps\to0$:
\begin{equation}
\label{eq:Kexpansion}
\begin{aligned}
\Ke(z_1,z_2)=&\;\Ie(\We)+O\left(\frac{\eps^3}{|\ln\eps|}\right)\\
\frac{\partial\Ke(z_1,z_2)}{\partial z_{i,h}}=&\;\frac{\partial\Ie(\We)}{\partial z_{i,h}}+O\left(\frac{\eps^{2+\sigma}}{|\ln\eps|}\right),
\end{aligned}
\end{equation}
uniformly on compact subsets of $\mathcal M$.
Furthermore,
\begin{equation}
\Ie(W_\eps)=-A_1\frac{\eps^2\ke}{\ln\frac{R}{s\eps}}\,\mathcal H_{\ae,\aae}(z_1,z_2)
+\eps^2\ke A_{2,\eps}-\eps^2\ke^2 A_{3,\eps}+O\left(\frac{\eps^3}{|\ln\eps|}\right),
\end{equation}
where $A_1,A_{2,\eps},A_{3,\eps}>0$ are the uniformly bounded constants defined by
\begin{equation}
\label{def:Ai}
\begin{aligned}
A_1=&\;\frac{1}{2}\int_{B_s(0)}\varphi_1,\\
A_{2,\eps}=&\;\frac{\ae^2+\aae^2}{2}\int_{B_s(0)}\varphi_1,\\
A_{3,\eps}=&\;\frac{\ae^2+\aae^2}{2}\int_{B_s(0)}\varphi_1^2.
\end{aligned}
\end{equation}
%
%
%
\end{prop}
In order to prove Proposition~\ref{prop:redfunct} we will use the following lemmata.
\begin{lemma}
\label{lem:gradUexp}
Let $(z,\zeta)\in\mathcal M$.
For any $a,b>0$ and for any sufficiently small $\eps>0$
the following expansions hold:
\begin{enumerate}
  \item [(i)]
$\int_\Omega|\nabla PU_{\eps,a,z}(x)|^2\,dx
=a^2\ke\left[\int_{B_s(0)}\varphi_1+\ke\int_{B_s(0)}\varphi_1^2-\frac{\int_{B_s(0)}\varphi_1}{\ln\frac{R}{s\eps}}g(z,z)\right]
+O\left(\frac{\eps}{|\ln\eps|^2}\right),$
\item[(ii)]
$\int_\Omega\nabla PU_{\eps,a,z}(x)\cdot\nabla PU_{\eps,b,\zeta}(x)\,dx 
=ab\,\ke\frac{\int_{B_s(0)}\varphi_1}{\ln\frac{R}{s\eps}}\bG(z,\zeta)+O\left(\frac{\eps}{|\ln\eps|^2}\right),$
\end{enumerate}
uniformly on compact subsets of $\mathcal M$.
\end{lemma} 
\begin{proof}
Proof of (i).
We compute, recalling \eqref{def:bubble} and Lemma~\ref{lem:proj-exp}:
\begin{align*}
\int_\Omega|\nabla PU_{\eps,a,z}(x)|^2\,dx=&-\int_\Omega(\Delta PU_{\eps,a,z})(x)PU_{\eps,a,z}(x)\,dx
=\frac{1}{\eps^2}\int_{B_{s\eps}(z)}(U_{\eps,a,z}-a)_+PU_{\eps,a,z}\,dx\\
=&\;\frac{1}{\eps^2}\int_{B_{s\eps}(z)}a^2\ke\varphi_1\left(\frac{x-z}{\eps}\right)
\left[1+\ke\varphi_1\left(\frac{x-z}{\eps}\right)-\frac{g(x,z)}{\ln\frac{R}{s\eps}}\right]\,dx\\
=&\;a^2\ke\int_{B_s(0)}\varphi_1(y)\left[1+\ke\varphi_1(y)-\frac{g(z+\eps y,z)}{\ln\frac{R}{s\eps}}\right]\,dy\\
=&\;a^2\ke\int_{B_s(0)}\varphi_1(y)\left[1+\ke\varphi_1(y)-\frac{g(z,z)}{\ln\frac{R}{s\eps}}+O\left(\frac{\eps}{|\ln\eps|}\right)\right]\,dy.
\end{align*}
Proof of (ii). Similarly, we compute:
\begin{align*}
\int_\Omega\nabla PU_{\eps,a,z}(x)&\cdot\nabla PU_{\eps,b,\zeta}(x)\,dx
=-\int_\Omega(\Delta PU_{\eps,a,z})(x)PU_{\eps,b,\zeta}(x)\,dx\\
=&\;\frac{1}{\eps^2}\int_{B_{s\eps}(z)}(U_{\eps,a,z}-a)_+\,PU_{\eps,b,\zeta}\,dx
=\frac{1}{\eps^2}\int_{B_{s\eps}(z)}a\ke\varphi_1\left(\frac{x-z}{\eps}\right)\frac{b}{\ln\frac{R}{s\eps}}\bG(x,\zeta)\,dx\\
=&\;\frac{ab\,\ke}{\ln\frac{R}{s\eps}}\int_{B_s(0)}\varphi_1(y)\bG(z+\eps y,\zeta)\,dy
=\frac{ab\,\ke}{\ln\frac{R}{s\eps}}\int_{B_s(0)}\varphi_1(y)(\bG(z,\zeta)+O(\eps|y|))\,dy\\
=&\;\frac{ab\,\ke}{\ln\frac{R}{s\eps}}\int_{B_s(0)}\varphi_1(y)\bG(z,\zeta)\,dy+O\left(\frac{\eps}{|\ln\eps|^2}\right).
\end{align*}
\end{proof}
\begin{lemma}
\label{lem:Wmasses}
The following expansions hold:
\begin{enumerate}
\item [(i)]
$\int_\Omega[(W_\eps-1)_+]^2\,dx=\eps^2a_1^2\ke^2\int_{B_s(0)}\varphi_1^2+O\left(\frac{\eps^3}{|\ln\eps|}\right);$
\item[(ii)] 
$\int_\Omega[(-W_\eps-\gamma)_+]^2\,dx=\eps^2a_2^2\ke^2\int_{B_s(0)}\varphi_1^2+O\left(\frac{\eps^3}{|\ln\eps|}\right).$ 
\end{enumerate}
The convergences are uniform on compact subsets of $\mM$.
\end{lemma}
\begin{proof}
We write:
\begin{align*}
\int_\Omega[(W_\eps-1)_+]^2\,dx=&\int_{B_{s\eps}(z_1)}[(W_\eps-1)_+]^2\,dx
+\int_{B_{s\eps}(z_2)}[(W_\eps-1)_+]^2\,dx\\
&\qquad+\int_{\Omega\setminus(B_{s\eps}(z_1)\cup B_{s\eps}(z_2))}[(W_\eps-1)_+]^2\,dx.
\end{align*}
In view of Lemma~\ref{lem:Wid} we compute:
\begin{align*}
\int_{B_{s\eps}(z_1)}[(W_\eps-1)_+]^2\,dx=&\int_{B_{s\eps}(z_1)}\left[a_1\ke\varphi_1\left(\frac{x-z_1}{\eps}\right)+O\left(\frac{|x-z_1|}{|\ln\eps|}\right)\right]^2\,dx\\
=&\;a_1^2\ke^2\int_{B_{s\eps}(z_1)}\varphi_1^2\left(\frac{x-z_1}{\eps}\right)\,dx+O\left(\frac{\eps^3}{|\ln\eps|^2}\right)\\
=&\;\eps^2a_1^2\ke^2\int_{B_{s}(0)}\varphi_1^2+O\left(\frac{\eps^3}{|\ln\eps|^2}\right).
\end{align*}
On the other hand, for sufficiently small $\eps>0$ we have
\[
\int_{B_{s\eps}(z_2)}[(W_\eps-1)_+]^2\,dx=0=\int_{\Omega\setminus(B_{s\eps}(z_1)\cup B_{s\eps}(z_2))}[(W_\eps-1)_+]^2\,dx.
\]
The remaining estimates are derived similarly.
\end{proof}
\begin{proof}[Proof of Proposition~\ref{prop:redfunct}]
In view of Lemma~\ref{lem:gradUexp} we have
\begin{equation}
\begin{aligned}
\int_\Omega|\nabla W_\eps|^2\,dx
=&\int_\Omega(|\nabla PU_1|^2-2\nabla PU_1\cdot\nabla PU_2+|\nabla PU_1|^2)\,dx\\
=&\;\frac{\ke}{\ln\frac{R}{s\eps}}\int_{B_s(0)}\varphi_1\left[-\ae^2g(z_1,z_1)-2\ae\aae\bG(z_1,z_2)-\aae^2g(z_2,z_2)\right]\\
&\qquad+\left(\ae^2+\aae^2\right)\ke\left(\int_{B_s(0)}\varphi_1+\ke\int_{B_s(0)}\varphi_1^2\right)+O\left(\frac{\eps}{|\ln\eps|^2}\right).
\end{aligned}
\end{equation}
We conclude the proof if the first part of \eqref{eq:Kexpansion} in view of Lemma~\ref{lem:Wmasses}.
\par
We are left to prove the second part of  \eqref{eq:Kexpansion}, namely the
$C^1$-approximation property.
The proof closely follows \cite[Section 3]{CaoPengYan} therefore we only outline the main steps.
\par
\textit{Step 1.} We first observe that $\omega_{\eps,Z}$ is continuous in $H^1$-norm.
We already know that $\omega_{\eps,Z}$ is uniformly bounded in $L^\infty(\Omega)$,
uniformly in $Z$ satisfying \eqref{assump:Z}.
The asserted continuity property follows by elliptic regularity applied to the projected problem
\[
-\eps^2\Delta(\We+\ome)-(\We+\ome-1)_+-(-\We-\ome-\ga)_+
=\sum_{j,h=1}^2b_{j,h}\xi\left(\frac{|y-z_j|}{s\eps}\right)\frac{\partial U_j}{\partial z_{j,h}}.
\]

\textit{Step 2.}
We prove that $\frac{\partial\omega_{\eps,Z}}{\partial z_{j,h}}$ is continuous in $H^1$-norm and 
\[
\left \|\frac{\partial\omega_{\eps,Z}}{\partial z_{j,h}}\right \|_\infty=O\left(\frac{1}{\eps^{1-\sigma}|\ln\eps|}\right).
\]
The proof is based on a differential quotients techniques and strongly relies on Proposition \ref{prop:zeromeasure}. More precisely, let $\bf{e}\in\mathbb R^4$ be a unit vector and let $s\neq0$, for any function $v_Z$ we denote the $s$-difference quotient of $v_Z$ in the direction $\bf{e}$ at $Z$ with
\[
\Delta^s  v_Z(y)=\frac{v_{Z+s\bf{e}}(y)-v_Z(y)}{d}.
\]
For any $a\in\mathbb R$, let $K_s^a:=\{y:d(y,\{u=a\})<s\}$.
Then,
\[
\Delta^s(u-a)_+=1_{\{u>a\}}\Delta^s u+O(1_{K_s^a}|\Delta^s u|).
\]
Now the crucial consequence of Proposition~\ref{prop:zeromeasure}
is that $mK_s^1\to0$ and $mK_s^{-\ga}\to0$ as $s\to0$.
A careful analysis as in the proof of \cite[Lemma 3.7]{CaoPengYan} together with elliptic regularity applied to $\Delta^s\ome$ yields the desired property
\[
\|D_{\bf{e}}\omega_{\eps,Z}\|_\infty\le\frac{C}{\eps^{1-\sigma}|\ln\eps|}.
\]
\end{proof}
\begin{proof}[Proof of Theorem~\ref{thm:main}]
We argue as in the proof of Theorem~1.3 in \cite{EspositoMussoPistoia2007}.
we equivalently seek critical points for the functional $\Fe$ defined by
\[
\Fe(z_1,z_2)=-\frac{\ln\frac{R}{s\eps}}{\eps^2\ke}A_1^{-1}\left(K_\eps(z_1,z_2)
-\eps^2\ke A_{2,\eps}+\eps^2\ke^2 A_{3,\eps}\right),
\]
where $A_1$, $A_{2,\eps}$ and $A_{3,\eps}$ are the constants defined in \eqref{def:Ai}.
In view of Proposition~\ref{prop:redfunct}, $\Fe$ is approximated by $\mathcal H_{\ae,\aae}$,
uniformly in $C^1$ on compact subsets of $\mM$.
Moreover, $\mathcal H_{\ae,\aae}(z_1,z_2)\to+\infty$ as $(z_1,z_2)\to\partial\mathcal M$.
Let $C\subset\mathcal M$ be a compact set such that $\mathrm{cat}C=\mathrm{cat}\mathcal M$.
Let $\mathcal U\subset\mathcal M$ be an open set such that $C\subset\mathcal U$
and
\[
\inf_{\partial\mathcal U}\mathcal H_{\ae,\aae}>\max_C\mathcal H_{\ae,\aae}.
\]
By taking $\eps>0$ sufficiently small, we derive
\[
\inf_{\partial\mathcal U}\Fe>\max_C\Fe.
\]
For $j=1,2,\ldots,\mathrm{cat}\mathcal M$ we set
\begin{align*}
c_\eps^j:=&\inf\{c:\ \mathrm{cat}_\mathcal M(\Fe^c\cap\mathcal U\ge j)\}\\
=&\inf\{\max\Fe:\ A\subset\mathcal U\ \mathrm{compact},\ \mathrm{cat}_\mathcal MA\ge j\},
\end{align*}
where $\Fe^c=\{(z_1,z_2)\in\mathcal M:\ \Fe(z_1,z_2)\le c\}$.
Now, standard arguments based on the Deformation Lemma
(see, e.g., \cite{Ambrosetti1992}, Theorem~2.3) imply
that $c_\eps^j$, $j=1,2,\ldots,\mathrm{cat}\mathcal M$,
are critical levels for $\Fe$.
Finally, we observe that
\[
c_\eps^1=\min_{\mathcal M}\Fe.
\]
See also Theorem~2.1 in \cite{BartschPis} and Theorem~1.1 in \cite{PistoiaRicciardi}.
\par
Proof of (i).
The proof follows by analogous arguments as in Proposition~\ref{prop:zeromeasure}.
\par
Proof of (ii).
If $\ga=1$, for every solution $u_\eps$ to \eqref{eq:pb1} we obtain a second solution
given by $-u_\eps$. Hence, the reults follows by standard symmetry arguments.
\end{proof}
\section{Qualitative properties of $\ue$ and proof of Theorem~\ref{thm:properties}}
\label{sec:qualprops}
We recall that the Kirchhoff-Routh type Hamiltonian $\Hg$
is defined by:
\[
\Hg(z_1,z_2)=h(z_1)+2\ga\bG(z_1,z_2)+\ga^2 h(z_2)
\]
for all $(z_1,z_2)\in\mathcal M$.
Let $(z_1^\ga,z_2^\ga)\in\mM$ be a minimum point for $\Hg$,
namely
\[
\Hg(z_1^\ga,z_2^\ga)=\min_{\mM}\Hg.
\]
We note that since $h(z)\to+\infty$ as $z\to\partial\Omega$ and $\bG(z_1,z_2)\to+\infty$
as $d(z_1,z_2)\to0$, the minimum of $\Hg$ is well-defined and it is attained in $\mM$.
\par
We recall that $\zmin\in\Omega$ is a \textit{harmonic center} for $\Omega$
if it is a minimum point for $h$, that is:
\begin{equation*}
h(\zmin)=\min_{\Omega}h \geq 0.
\end{equation*}
For every $\ga\in(0,1)$ let $(\zg,\zgg)$ be a minimum point for $\Hg$, that is:
\begin{equation*}
\Hg(\zg,\zgg)=\min_{\mM}\Hg.
\end{equation*}
Up to passing to a subsequence, we may assume that $(\zg,\zgg)\to(z_1^0,z_2^0)\in\overline\Omega\times\overline\Omega$
as $\gamma\to0^+$.
The following proposition contains the main ingredients needed for the proof of Theorem~\ref{thm:properties}.
\begin{prop}
\label{lem:Hamiltonianmax}
The following results hold true.
\begin{enumerate}
  \item[{(i)}]
$z_1^0\in\Omega$; moreover, $z_1^0$ is a harmonic center for $\Omega$.
\item[{(ii)}]
$z_2^0\in\partial\Omega$; moreover $\partial_\nu(z_1^0,z_2^0)=\max_{p\in\partial\Omega}G_\nu(z_1^0,p)$,
where $\nu$ denotes the outer normal direction to $\partial\Omega$. 
\end{enumerate}
\end{prop}
\begin{proof}
We begin by proving that for every $\eta>0$ there exist a harmonic center $\zmin\in\Omega$
and  $0<\ga_\eta\ll1$ such that
\begin{equation*}
\begin{aligned}
&d(z_1^\ga,\zmin)<\eta,
&&d(z_2^\ga,\partial\Omega)<\eta.
\end{aligned}
\end{equation*}
To this end, we first check that there exists $M>0$ independent of $\ga$ such that
\begin{equation}
\label{eq:minHgupperbound}
\min_\mM\Hg\le M.
\end{equation}
Indeed, fix any $(\bar z_1,\bar z_2)\in\mM$.
Then, taking into account of \eqref{eq:hproperties},
\begin{equation*}
\begin{aligned}
\min_{\mM}\Hg\le&\;\Hg(\bar z_1,\bar z_2)
=h(\bar z_1)+2\ga\bG(\bar z_1,\bar z_2)+\ga^2h(\bar z_2)\\
\le&\;h(\bar z_1)+2\bG(\bar z_1,\bar z_2)+h(\bar z_2)=:M.
\end{aligned}
\end{equation*}
\par
\textit{Claim 1.} $z_1^0\not\in\partial\Omega$.
\par
Arguing by contradiction, we assume that $\zg\to\partial\Omega$ as $\ga\to0^+$.
We deduce that 
\begin{align*}
\min_{\mM}\Hg=&\;\Hg(\zg,\zgg)
=h(\zg)+2\ga\bG(\zg,\zgg)+\ga^2h(\zgg)\\
\ge&\;h(\zg)\to+\infty,
\end{align*}
a contradiction to \eqref{eq:minHgupperbound}. We conclude that $z_1^0\in\Omega$.
\par
\textit{Claim 2.} $z_2^0\neq z_1^0$.
\par
Arguing by contradiction, assume that $d(\zg,\zgg)\to0$ as $\ga\to0^+$.
We deduce that
\begin{align*}
\min_{\mM}\Hg=&\;\Hg(\zg,\zgg)
\ge h(\zg)+2\ga\bG(\zg,\zgg)+\min_\Omega h\\
\ge&\;h(z_1^0)+2\ga\bG(\zg,\zgg)+\min_\Omega h-1\to+\infty,
\end{align*}
a contradiction to \eqref{eq:minHgupperbound}.
\par
\textit{Claim 3.}
$d(\zgg,\partial\Omega)\to0$.
\par
\par
Arguing by contradiction and we assume that there exists $\eps_0>0$ such that 
$d(z_2^\ga,\partial\Omega)\ge2\eps_0$ for all $\ga\in(0,\ga_\eta')$.
For every $x\in\Omega$ let $d_x=d(x,\partial\Omega)$.
Without loss of generality, we may assume that the tubular neighborhood
\[
\Omega_{2\eps_0}=\{x\in\Omega:\ d_x<2\eps_0\}
\]
is well-defined.
We denote by $\widetilde\pi_{\eps_0}$ the projection $\widetilde\pi_{\eps_0}:\Omega_{2\eps_0}\to\partial\Omega_{\eps_0}\setminus\partial\Omega$.
In view of the Hopf maximum principle we may also assume that for any $\zeta\in\partial\Omega$
the function
\[
\psi(t):=\bG(\zg,\zeta-t\nu_\zeta),
\]
where $\nu_\zeta$ denotes the outer normal at $\zeta$, is strictly increasing for $t\in[0,2\eps_0]$.
Moreover, applying the maximum principle to the harmonic function $\bG(\zg,\zgg)$ in $(\Omega\setminus B_\rho(\zg))\setminus\Omega_{2\eps_0}$,
where $\rho>0$ is a small constant, we find $\tilzgg\in\partial\Omega_{2\eps_0}\setminus\partial\Omega$ such that
\[
\inf_{\Omega\setminus\Omega_{2\eps_0}}\bG(\zg,\cdot)=\inf_{(\Omega\setminus B_\rho(\zg))\setminus\Omega_{2\eps_0}}\bG(\zg,\cdot)
=\inf_{\partial\Omega_{2\eps_0}\setminus\partial\Omega}\bG(\zg,\cdot)=\bG(\zg,\tilzgg).
\]
The strong maximum principle also implies that $\partial_\nu\bG(z_1^0,\cdot)\le -\alpha_0$
on $\partial\Omega$, for some $\alpha_0$. 
Consequently, there exists $\beta(\eps_0)$ such that  
\[
\bG(z_1^\ga,\tilzgg)-\bG(z_1^\ga,\widetilde\pi_{\eps_0}\tilzgg)\ge\beta(\eps_0).
\]
We conclude that
\begin{equation*}
\begin{aligned}
\Hg(z_1^\ga,z_2^\ga)
=&\;\min_\mM\Hg=h(\zg)+2\ga\bG(\zg,\zgg)+\ga^2h(\zgg)
\ge h(\zg)+2\ga\bG(\zg,\tilzgg)+\ga^2h(\zgg)\\
=&\;\Hg(z_1^\ga,\pi_{\eps_0}\tilzgg)
+\ga\{2[\bG(z_1^\ga,\tilzgg)-\bG(z_1^\ga,\pi_{\eps_0}\tilzgg)]+\ga[h(z_2^\ga)-h(\widetilde\pi_{\eps_0}\tilzgg)]\}.
\end{aligned}
\end{equation*}
Let $0<\ga_\eta''\le\ga_\eta'$ be such that
\[
\ga|h(z_2^\ga)-h(\widetilde\pi_{\eps_0}\tilzgg)|<\beta(\eps_0)
\]
for all $0<\ga<\ga_\eta''$.
Then, for all $0<\ga<\ga_\eta''$ we obtain that
\[
\Hg(z_1^\ga,z_2^\ga)=\min_{\mM}\Hg
>\Hg(z_1^\ga,\widetilde\pi_{\eps_0}\tilzgg)+\ga\beta(\eps_0)>\Hg(z_1^\ga,z_2^\ga),
\]
a contradiction.
\par
\textit{Claim 4.}
$h(z_1^0)=\min_\Omega h$.
\par
Arguing by contradiction, assume that $h(\zg)-\min_\Omega h\ge\eta_0>0$
for all sufficiently small values of $\ga$.
We write
\begin{align*}
\Hg(\zg,\zgg)=&\;h(\zg)+2\ga\bG(\zg,\zgg)+\ga^2 h(\zgg)\\
=&\;h(\zmin)+2\ga\bG(\zmin,\zgg)+\ga^2h(\zgg)
+h(\zg)-h(\zmin)
+2\ga[\bG(\zg,\zgg)-\bG(\zmin,\zgg)]\\
\ge&\;\Hg(\zmin,\zgg)+\eta_0+2\ga[\bG(\zg,\zgg)-\bG(\zmin,\zgg)],
\end{align*}
where $\zmin\in\Omega$ is a harmonic center for $\Omega$.
For $\ga$ sufficiently small we have $2\ga|\bG(\zg,\zgg)-\bG(\zmin,\zgg)|\le\eta_0/2$.
It follows that
\begin{align*}
\Hg(\zg,\zgg)\ge\Hg(\zmin,\zgg)+\frac{\eta_0}{2}\ge\min_{\mM}\Hg+\frac{\eta_0}{2},
\end{align*}
a contradiction.
We conclude that $z_1^0$ is a harmonic center for $\Omega$,
and Part~(i) is established.
\par
\textit{Claim 6.}
We denote by $\pi_{\eps_0}:\Omega_{\eps_0}\to\partial\Omega$ the standard projection.
We recall from \cite{BF}, p.~204 that the expansion
\begin{equation}
\label{eq:hexp}
h(z-d_z\nu(z))=\frac{1}{2\pi}\log(2d_z)+o(d_z)
\end{equation}
holds true for every $z\in\Omega_{\eps_0}$.
We also note that in view of the mean value theorem we may write
\begin{equation}
\label{eq:Gmv}
\bG(\zg,z)=\bG(\zg,\pi_{\eps_0}z-d_z\nu)=-\partial\bG_\nu(\zg,\pi_{\eps_0}z)d_z+o(d_z).
\end{equation}
Let $p\in\partial\Omega$. Then, $\Hg(\zg,\zgg)\le\Hg(\zg,p-d_{\zgg}\nu(p))$
and we deduce that
\begin{align*}
2\bG(\zg,\zgg)+\ga h(\zgg)\le2\bG(\zg,p-d_{\zgg}\nu(p))+\ga h(p-d_{\zgg}\nu(p))
\end{align*}
In view of \eqref{eq:hexp} and using $\zgg=\pi_{\eps_0}\zgg-d_{\zgg}\nu(\zgg)$, we derive
\[
\bG(\zg,\zgg)+\frac{\ga}{2}\left[\frac{1}{2\pi}\log(2d_{\zgg})+o(d_{\zgg})\right]
\le\bG(\zg,p-d_{\zgg}\nu(p))+\frac{\ga}{2}\left[\frac{1}{2\pi}\log(2d_{\zgg})+o(d_{\zgg})\right].
\]
from which we derive that
\[
\bG(\zg,\zgg)\le\bG(\zg,p-d_{\zgg}\nu(p))+o(d_{\zgg}).
\]
Finally, in view of \eqref{eq:Gmv} we deduce
\[
-\partial\bG_\nu(\zg,\pi_{\eps_0}\zgg)d_{\zgg}\le\partial_\nu\bG(\zg,p)d_{\zgg}+o(d_{\zgg})
\]
and finally
\[
\partial\bG_\nu(\zg,\pi_{\eps_0}\zgg)\ge\partial_\nu\bG(\zg,p).
\]
Letting $\ga\to0^+$ we conclude that
\[
\partial\bG_\nu(z_1^0,z_2^0)\ge\partial_\nu\bG(z_1^0,p)
\]
for any $p\in\partial\Omega$.
Now Part~(ii) is completely established.
\end{proof}
Now we can conclude the proof of Theorem~\ref{thm:properties}.
\begin{proof}[Proof of Theorem~\ref{thm:properties}-(ii)]
We observe that the reduced functional $\Ke(z_1,z_2)$ is of the form
\[
\Ke(z_1,z_2)=\Ie(\We)=C_1(\eps)a_1^2\mathcal H_{a_2/a_1}(z_1,z_2)+C_2(\eps)
\]
and that
\[
\frac{a_2}{a_1}=\frac{\ga+O\left(\frac{1}{|\ln\eps|}\right)}{1+O\left(\frac{1}{|\ln\eps|}\right)}
=\ga+O\left(\frac{1}{|\ln\eps|}\right).
\]
In view of Lemma~\ref{lem:Hamiltonianmax} we conclude that 
for any $\eta>0$ there exist a sufficiently small $\ga_0>0$
with the property that for any $\ga\in(0,\ga_0)$ there exists $\eps_\ga>0$
such that the peak points $(z_1,z_2)$ of the constructed solution $\ue$
satisfy $d(z_1,\zmin)<\eta$ and $d(z_2,\partial\Omega)<\eta$
for all $\eps\in(0,\eps_\ga)$.
\end{proof}


\begin{thebibliography}{99}
\bibitem{Ambrosetti1992}
Ambrosetti, A.,
{\em Critical points and nonlinear variational problems.}
M\'em.\ Soc.\ Math.\ Fr.\ S\'er.~2 \textbf{49} (1992), 1--139.
\bibitem{Bandle}
Bandle, C.,
Isoperimetric Inequalities and Applications.
Pitman, London, 1980.
\bibitem{BF}
Bandle, C., Flucher, M.,
\textit{Harmonic Radius and Concentration of Energy; Hyperbolic Radius
and Liouville's Equations $\Delta U=e^U$ and $\Delta U=U^{\frac{n+2}{n-2}}$.}
SIAM Review \textbf{38} (1996), 191--238.
\bibitem{BarPis}  
Bartolucci, D., Pistoia, A. 
{\em  Existence and qualitative properties of concentrating solutions
for the sinh-Poisson equation.} 
IMA Journal of Applied Mathematics {\bf 72} (2007) 706--729.
\bibitem{BartschPis}
Bartsch, T., Pistoia, A.,
{\em Critical Points of the $N-$vortex Hamiltonian in bounded planar domains and steady state solutions
of the incompressible Euler equations.}
SIAM J.\ Appl.\ Math.\ \textbf{75} (2015) n.~2, 726--744.
\bibitem{BerestyckiBrezis}
Berestycki, H., Brezis, H.,
{\em On a free boundary problem arising in plasma physics.}
Nonlinear Analysis \textbf{4} (1980), 415--436.
\bibitem{CaffarelliFriedman}
Caffarelli, L., Friedman, A.,
{\em Asymptotic estimates for the plasma problem.}
Duke Math.\ J., \textbf{47} (1980) 705--742.
\bibitem{CaffarelliFriedman1985}  
Caffarelli, L., Friedman, A.,
{\em Convexity of solutions of semilinear elliptic equations,}
Duke.\ Math.\ J. \textbf{52} (1985), 431--457.
\bibitem{CaoLiuWei}
Cao, D., Liu, Z., Wei, J.,
{\em Regularization of Point Vortices Pairs for the Euler Equation in Dimension Two.}
Arch.\ Rat.\ Mech.\ Anal.\ \textbf{212} (2014), 179--217.
\bibitem{CaoPengYan}  
Cao, D., Peng, S., Yan, S. 
{\em  Multiplicity of solutions for the plasma problem in two dimensions.} 
Advances in Mathematics {\bf 225} (2010) 2741--2785.
\bibitem{CaoPengYan2015}  
Cao, D., Peng, S., Yan, S. 
{\em  Planar vortex patch problem in incompressible steady flow.} 
Advances in Mathematics {\bf 270} (2015) 263--301.
\bibitem{CHP}
Croce, G., Henrot, A., Pisante, G.,
{\em An isoperimetric inequality for a nonlinear eigenvualue problem.}
Ann.\ I.~H.~Poincar\'e - AN \textbf{29} (2012), 21--34.
\bibitem{CHP1}
Croce, G., Henrot, A., Pisante, G., 
{\em Corrigendum to \lq\lq An isoperimetric inequality for a nonlinear eigenvalue problem" 
[Ann. I. H. Poincar\'e -- AN 29 (1) (2012) 21--34]} 
Ann.\ Inst.\ H.~Poincar\'e AN \textbf{32} (2015), 485--487.
\bibitem{DancerYan}
Dancer, E.N., Yan, S.,
{\em The Lazer-McKenna conjecture and a free boundary problem in two dimensions.}
J.\ London Math.\ Soc.\ \textbf{78} n.~2 (2008), 639--662.
\bibitem{EspositoMussoPistoia2007}
Esposito, P., Musso, M., Pistoia, A.,
{\em On the existence and profile of nodal solutions for a two-dimensional elliptic problem
with large exponent in nonlinearity.}
Proc.\ London Math.\ Soc.\ \textbf{94} no.~3 (2007), 497--519.
\bibitem{HartmanWintner}
Hartman, P., Wintner, A.,
{\em On the local behavior of non-parabolic partial differential equations.}
Am.\ J.\ Math.\ \textbf{85} (1953), 449--476.
\bibitem{kinderlehrerspruck}
Kinderlehrer, D., Spruck, J.,
{\em The shape and smoothness of stable plasma configurations.}
Ann.\ Sc.\ Norm.\ Pisa Ser.~IV (1978), 113--148.
\bibitem{MallierMaslowe}
Mallier, R., Maslowe, S.A.,
{\em A row of counter rotating vortices.}
Phys.\ Fluids A \textbf{5} (1993), 1074--1075.
\bibitem{PistoiaRicciardi}
Pistoia, A., Ricciardi, T.,
{\em Concentrating solutions for a Liouville type equation with variable intensities in 2D-turbulence.}
To appear in Nonlinearity.
arXiv:1505.05304.
\bibitem{SmVS}  
Smets, D., Van Schaftingen, J.,
{\em  Desingularization of Vortices for the Euler Equation.} 
Arch.\ Rational Mech.\ Anal.\ {\bf 198} (2010) 869--925.
\bibitem{Te}
Temam, R.,
{\em A nonlinear equilibrium problem: The shape at equilibrium of a confined plasma.}
Arch.\ Rat.\ Mech.\ Anal.\ \textbf{60} (1975), 51--73.
\end{thebibliography}
\end{document}